\documentclass[10pt]{article}

\PassOptionsToPackage{hyphens}{url}\usepackage{hyperref}
\usepackage{amsthm,amsfonts,amsmath,amssymb,graphicx,enumerate,authblk,cleveref,listings}

\theoremstyle{theorem}
\newtheorem{theorem}{Theorem}[section]
\newtheorem{corollary}[theorem]{Corollary}
\newtheorem{lemma}[theorem]{Lemma}

\theoremstyle{definition}
\newtheorem{definition}[theorem]{Definition}
\newtheorem{example}[theorem]{Example}
\newtheorem{proposition}[theorem]{Proposition}
\newtheorem{remark}[theorem]{Remark}

\usepackage[colorinlistoftodos]{todonotes}
\begin{document}

\title{Simplicial cell decompositions of $\mathbb{CP}^{\hspace{.3mm}n}$}

\author[1] {Basudeb Datta}
\author[2] {Jonathan Spreer}

\affil[1]{\small Department of Mathematics, Indian Institute of Science, Bangalore 560\,012, India. \newline Institute for Advancing Intelligence, TCG CREST, Kolkata 700\,091, India. \newline Emails: {\it dattab@iisc.ac.in, bdatta17@gmail.com}.}
\affil[2]{School of Mathematics and Statistics, 
The University of Sydney, NSW 2006, Australia. Email: {\it jonathan.spreer@sydney.edu.au.}}

\date{We dedicate this work to Frank H. Lutz who left us too soon}

\maketitle

\vspace{-8mm}

\begin{abstract}
 According to a well-known result in geometric topology, we have \linebreak 
 $\left (\mathbb{S}^2 \right)^{n}\!\!/\operatorname{Sym}(n) = \mathbb{CP}^{n}$, where $\operatorname{Sym}(n)$ acts on $\left (\mathbb{S}^2 \right)^{n}$ by coordinate permutation. We use this fact to explicitly construct a regular simplicial cell decomposition of $\mathbb{CP}^{n}$ for each $n \geq 2$. In more detail, we start with the standard two triangle crystallisation $S^2_3$ of the $2$-sphere $\mathbb{S}^2$, in its $n$-fold Cartesian product. We then construct a simplicial subdivision of this product and prove that the $\operatorname{Sym}(n)$ quotient of this subdivision yields a simplicial cell decomposition of $\mathbb{CP}^n$. The first derived subdivision of this cell complex is a simplicial triangulation of $\mathbb{CP}^n$. To the best of our knowledge, this is the first explicit description of triangulations of $\mathbb{CP}^n$ for $n \geq 4$. 
\end{abstract}



\section{Introduction}

A long-standing problem in combinatorial and geometric topology is to find an explicit triangulation of the $n$-dimensional complex projective space $\mathbb{CP}^{n}$ using as few vertices as possible. In \cite{AM91}, Arnoux and Marin proved that any triangulation of $\mathbb{CP}^{n}$ requires at least $(n+1)^2$ vertices, and equality can only hold for $n\leq 2$. For $n=1$, $\mathbb{CP}^{1}\cong \mathbb{S}^{2}$ has a unique four-vertex triangulation $S^{2}_4$, the boundary complex of the three dimensional simplex. For $n=2$, it has long been known that $\mathbb{CP}^{2}$ has a unique 9-vertex triangulation $\mathbb{CP}^{2}_9$, first found by K\"{u}hnel and Banchoff (see \cite{AM91, BD94, KB83}). But the problem of constructing triangulations for larger $n$ that are as close to the Arnoux-Marin bound as possible remains open. Even more, for $n\geq 4$, no explicit triangulation of $\mathbb{CP}^{n}$ appears to be known. A representation of $\mathbb{CP}^{n}$ - as a quotient of a real $2n$-dimensional torus in $\mathbb{C}^{n+1}$ by $\operatorname{Sym}(n+1)$ - was used in \cite[\S\,14]{AM91} to obtain an $(n+1)^2$ vertex weak cellulation of $\mathbb{CP}^{n}$. The authors of \cite{AM91} observed that, when $n\geq 3$, this weak cellulation can not be triangulated without adding extra vertices. Triangulations of $\mathbb{CP}^{n}$ were announced in \cite{S2014}, but have since been withdrawn from arXiv.

One idea to obtain reasonably small triangulations of $\mathbb{CP}^{n}$ is based on the following observation: There exists a natural (non-free) action of the symmetric group $\operatorname{Sym}(n)$ on $(\mathbb{S}^{2})^n$, the $n$-th Cartesian power of the $2$-sphere, by coordinate permutation. It is well-known that the quotient space $(\mathbb{S}^{2})^n/\operatorname{Sym}(n)$ is the $n$-dimensional complex projective space $\mathbb{CP}^{n}$ (see \cite[Lemma 2.3]{BD12}). A triangulation of complex projective space is then obtained by triangulating the branched covering $(\mathbb{S}^{2})^n \to (\mathbb{S}^{2})^n/\operatorname{Sym}(n)= \mathbb{CP}^{n}$.

This approach is explored in \cite{BD11, BD12}: Starting with the 4-vertex triangulation $S^2_4$ of $\mathbb{S}^2$, a simplicial subdivision $((S^2_4)^n)^{\prime}$ of the cell complex $(S^2_4)^n$ is considered such that the quotient $((S^2_4)^n)^{\prime}/\operatorname{Sym}(n)$ is a triangulation of $\mathbb{CP}^n$. Carrying out this method yields a $10$-vertex triangulation of $\mathbb{CP}^{2}$ and a $30$-vertex triangulation of $\mathbb{CP}^{3}$. The latter triangulation can be reduced to an $18$-vertex triangulation of $\mathbb{CP}^{3}$ by using bistellar flips (specifically, the GAP program BISTELLAR due to Lutz \cite{GAP,L03}, or any other implementation of simplification heuristics in mathematical software, such as \cite{regina,simpcomp}). 
Note that, for $n\geq 3$, it is {\em not} possible to simplicially subdivide $(S^2_4)^n$ without adding extra vertices such that the quotient is a triangulation of a manifold, cf. \cite{BD12}. 

In \cite{KS23}, the authors follow an alternative approach to find triangulations of projective spaces: in the complex case, $\mathbb{CP}^n$ is decomposed into $(n+1)$ balls of real dimension $2n$ as ``zones of influence'', where a given homogeneous coordinate direction is dominant in absolute value. Sets of these balls naturally intersect in points that have the corresponding set of coordinates jointly being dominant. The intersection of all balls yields a central $n$-dimensional torus. This decomposition of $\mathbb{CP}^n$ into $2^{n+1}-1$ subsets is known as its {\em equilibrium decomposition}. This point of view gives one way to triangulate $\mathbb{CP}^n$, by first starting with a triangulation of the $n$-torus, and then ``filling in'' more and more bits and pieces. This approach leads to a beautiful $10$-vertex triangulation of $\mathbb{CP}^2$, first described in \cite{BK92}, but faces problems already for $n=3$.

Other approaches to construct triangulations of $\mathbb{CP}^n$, including a sketch of the approach presented here, are discussed in \cite{overflow}. Moreover, there exists an overview over triangulations of complex projective spaces in low dimensions due to Sergeraert \cite{S2010}, where the notion of {\em triangulation} denotes a more general type of decomposition with non-bijective gluing maps. 

While triangulations of $\mathbb{CP}^n$ are notoriously difficult to explicitly write down, this is easier in the real case of $\mathbb{RP}^n$. Here, the first family of explicit triangulations is due to K\"{u}hnel \cite{K1986}. The fundamental challenge that remains in the real case, is to find triangulations with as few vertices as possible. The classical triangulations from \cite{K1986} have $2^{n+1}-1$ vertices, but recent advances include triangulations with $O\left(\left(\frac{1+\sqrt{5}}{2}\right)^{n+1}\right)$ vertices \cite{VZ2021} and even a family with number of vertices subexponential in $n$ \cite{AAK2022}. The question of whether there exist triangulations of $\mathbb{RP}^n$ with number of vertices a polynomial in $n$ is still open. Moreover, leaving the realm of simplicial complexes and considering the more general class of (regular) simplicial cell complexes, a very elegant and simple construction exists producing such complexes for all $\mathbb{RP}^n$ with the smallest number of vertices, see \Cref{exam:RP^n}.

In this article, we follow exactly this strategy for triangulations of $\mathbb{CP}^n$. Namely, we again exploit the action of $\operatorname{Sym}(n)$ on $(\mathbb{S}^2)^n$, but apply it to (regular) simplicial cell decompositions of $(\mathbb{S}^{2})^n$ and $\mathbb{CP}^{n}$ (rather than simplicial complexes): 
Let $S^2_3$ be the simplicial cell complex consisting of two triangles, glued along their boundaries. Taking the $n$-fold Cartesian product of $S^2_3$ with itself, we obtain a cell decomposition $\widetilde{X}^n$ of $(\mathbb{S}^2)^n$ with $2^n$ cells of real dimension $2n$. We introduce the notion of a {\em good action} of a group $G\leq \operatorname{Aut}(X)$ on a simplicial cell complex $X$, and prove that if such an action is good, then $X/G$ is also a simplicial cell complex and $|X/G|$ is homeomorphic to $|X|/G$ (cf. Lemma \ref{lem:quotient} and Corollary \ref{cor:good}). We then construct a simplicial 
subdivision $X^n$ (Theorem \ref{theo:scdX^n}) of $\widetilde{X}^n$ without adding new vertices such that the action of $\operatorname{Sym}(n)$ on $X^n$ is good (cf. Theorem \ref{theo:good}). The result is a simplicial cell structure $X^n/\operatorname{Sym}(n)$ of $\mathbb{CP}^n$ for all $n$ (cf. Theorem \ref{theo:cp^n}). In fact, it can be shown that $X^n$ can be represented as an edge coloured version of its dual graph, and thus is a graph encoded manifold, see \cite{FGG1986} for details.
Taking the first derived subdivision of the simplicial cell complex $X^n/\operatorname{Sym}(n)$ one obtains a triangulation of $\mathbb{CP}^n$ for each $n \geq 2$. 
This yields, for the first time, an explicit triangulation of $\mathbb{CP}^n$ for each $n\geq 4$. 

Code to produce simplicial cell complexes $X^n$ and $X^n/\operatorname{Sym}(n)$ is discussed in \Cref{sec:code,app:code}, and given in \cite{DS2024}. In \Cref{app:isosigs}, we list $X^n/\operatorname{Sym}(n)$, $n \in \{2,3,4\}$ explicitly as {\em Regina} \cite{regina} triangulations and/or graph encodings.

\subsection*{Acknowledgements}

Research of the second author is supported in part under the Australian Research Council's Discovery funding scheme (project number DP220102588). This article was completed while the second author was on sabbatical at the Technische Universit\"at (TU) Berlin. The author would like to thank TU Berlin, and particularly Michael Joswig and his research group, for their hospitality in times of difficult circumstances. The authors thank the anonymous referees for some useful comments which led to improvements in the presentation of this paper.

\section{Preliminaries}

By a poset, we mean  a finite $X$ with a partial ordering $\leq$. By $\alpha<\beta$, we mean $\alpha\leq \beta$ and $\alpha\neq \beta$. A bijection $f: (K, \leq) \to (L, \leq)$ between two posets is called an {\em isomorphism} if $a<b$ if and only if $f(a) < f(b)$ for $a, b\in K$. Two posets are called {\em isomorphic} if there exits an isomorphism between them. In this article, we do not distinguish between isomorphic posets.  An isomorphism $f : (K, \leq) \to (K, \leq)$ is called an {\em automorphism} of $(K, \leq)$. The set $\operatorname{Aut}(K)$ of all the automorphisms of $(K, \leq)$ is a group under composition and is called the {\em automorphism group} of $(K, \leq)$. Any subgroup of $\operatorname{Aut}(K)$ is called a {\em group of automorphisms} of $K$. 

A CW-complex is said to be {\em regular} if all of its closed cells are homeomorphic to closed balls. Given a CW-complex $M$, let ${\mathcal M}$ be the set of all closed cells of $M$ together with the empty set $\emptyset$. Then ${\mathcal M}$ is a poset, where the partial ordering is given by set inclusion. This poset ${\mathcal M}$ is said to be the {\em face poset} of $M$. If $M$ and $N$ are two finite regular CW-complexes with isomorphic face posets then $M$ and $N$ are homeomorphic. A regular CW-complex $M$ is said to be {\em simplicial}, if the {\em boundary complex} $\{\alpha\in {\mathcal M} : \alpha \subsetneq \sigma\}$ of each cell $\sigma$ of $M$ is isomorphic (as a poset) to the boundary complex of a simplex. A simplicial CW-complex is also called a {\em simplicial cell complex}. All CW-complexes considered in this paper are finite.

A {\em simplicial poset} $K= (K,\leq)$ of dimension $d$ is a poset isomorphic to the face poset ${\mathcal M}$ of a $d$-dimensional simplicial cell complex $M$. The topological space $M$ is called the {\em geometric carrier} of $K$ and is also denoted by $|K|$. If a topological space $N$ is homeomorphic to $|K|$, then $K$ is said to be a {\em simplicial cell decomposition} (or {\em pseudotriangulation}) {\em of $N$}. If $|K|$ is connected then we say $K$ is {\it connected}. Simplicial posets are also referred to in the literature as {\em CW posets}, {\em pseudo-simplicial complexes} or {\em pseudocomplexes} (see \cite{BD14, B84, M13} for details).  We identify a simplicial cell complex $M$ with the face poset of $M$. Accordingly, by a simplicial cell complex we also mean a simplicial poset. 

For $\alpha, \beta\in K$, we say $\alpha$ is {\em incident} to $\beta$ if $\alpha\leq\beta$ or $\beta\leq\alpha$. For $\alpha\in K$, $\partial \alpha := \{\beta \in K \, : \, \beta <\alpha\}$ is a simplicial poset and is said to be the {\em boundary} of $\alpha$. If $\partial \alpha$ is isomorphic to the boundary complex of a $j$-simplex, then $\alpha$ is a called a {\em $j$-cell} of $K$. The 0-cells and 1-cells of a simplicial poset $K$ are called the {\it vertices} and {\em edges} of $K$ respectively. The set $V(K)$ of vertices of a simplicial poset $K$ is called the {\em vertex set} of $K$. The vertex set $V_K(\alpha)$ ($:=\{v\in V(K) : v\leq\alpha\}$) of a $j$-cell $\alpha$ has $j+1$ vertices. If $U\subseteq V_K(\alpha)$, then the (unique) face of $\alpha$ whose vertex set is $U$ is denoted by $\alpha(U)$. 

For  an element $\alpha$ of a simplicial poset $K$, the poset $\operatorname{Lk}_{K}(\alpha) := \{\beta\in K \, : \, \alpha \leq \beta\}$ is also a simplicial poset and is called the {\em link} of $\alpha$ in $K$ (cf. \cite{M13}). Observe that $\alpha$ is the smallest element (i.e., the unique element of dimension $-1$) in $\operatorname{Lk}_{K}(\alpha)$.  Note that for this definition to work it is essential that the cell complex underlying $K$ is regular.

To gain some geometric intuition about the link, suppose $K$ is the face poset of a simplicial CW-complex $M$. We can assume that $M$ is a metric space. For each vertex $u$ of $K$, consider a sphere $S_u$ in $M$ with centre $u$ and of small radius so that $S_u$ intersects only members of $\operatorname{Lk}_K(u)\setminus\{u\}$ and $u$ is the only vertex interior of $S_u$.  Then $L_u := \{\alpha\cap S_u \, : \, \alpha\in M\}$ is a simplicial CW-complex, and for each $j$-cell $\beta$ of $M$ with $u\leq\beta$, $S_u\cap\beta$ is a $(j-1)$-cell of $L_u$. This shows that $\operatorname{Lk}_K(u)$ is isomorphic to the face poset of $L_u$. Now suppose that $\alpha$ is a $j$-cell, $j\geq 1$. Let $u$ be a vertex of $\alpha$. Then $\alpha$ is a $(j-1)$-cell in $\operatorname{Lk}_K(u)$ and $\operatorname{Lk}_K(\alpha)= \operatorname{Lk}_{\operatorname{Lk}_K(u)}(\alpha)$. By induction on the dimension of the cell $\alpha$, $\operatorname{Lk}_{\operatorname{Lk}_K(u)}(\alpha)$ is a simplicial cell complex and hence $\operatorname{Lk}_K(\alpha)$ is so. 

Suppose $L\subseteq K$ are two simplicial posets. We say that $L$ is a {\it sub-poset} of $K$ if $\beta$ is a $j$-cell of $K$ and $\beta\in L$ implies the dimension of $\beta$ in $L$ is also $j$. So, $L$ is a sub-poset of $K$ $\Longleftrightarrow$ $(\gamma < \beta$ and $\gamma, \beta\in K$ imply $\gamma\in L)$. Equivalently, if $K$ is the face poset of a simplicial CW-complex $M$ then $L$ is the face poset of a subcomplex $N$ of $M$. Thus, for $\alpha\in K$, $\partial\alpha$ is a sub-poset of $K$ but $\operatorname{Lk}_K(\alpha)$ is not for $0\leq \dim(\alpha) < \dim(K)$. (Suppose, $\emptyset< \alpha < \beta$ are cells of a simplicial poset $K$. Then $\{\emptyset\} \subsetneq \operatorname{Lk}_K(\alpha)\subsetneq K$. If $\gamma < \beta$ and $\alpha\not\leq \gamma$, then $\beta\in \operatorname{Lk}_K(\alpha)$, $\gamma < \beta$ in $K$ but $\gamma\not\in\operatorname{Lk}_K(\alpha)$. Thus $\operatorname{Lk}_K(\alpha)$ is not a sub-poset of $K$.) 

If $K$ is a $d$-dimensional simplicial poset and $k\leq d$ then $\operatorname{skel}_k(K) := \{\alpha\in K \, : \, \dim(\alpha) \leq k\}$ is a sub-poset of $K$ (with induced partial ordering) of dimension $k$ and is called the {\em $k$-skeleton} of $K$. The $1$-skeleton of $K$ is a graph and is called the {\em edge graph} of $K$. $K$ is connected if and only if the graph $\operatorname{skel}_1(K)$ is connected. We say two vertices are adjacent in $K$ if they are adjacent in $\operatorname{skel}_1(K)$. If a vertex $v$ is a vertex in $k$ edges of $K$ then we say $v$ is a vertex of {\it degree $k$} and we write $\deg_K(v)=k$. 

A $d$-dimensional simplicial poset $K$ is called {\it pure} if each cell of $K$ is a face of a $d$-cell (i.e., all the maximal cells are $d$-cells). The $d$-cells and $(d-1)$-cells of a pure $d$-dimensional simplicial poset $K$ are called the {\it facets} and {\it ridges} of $K$ respectively. If $K$ is pure and $\alpha\in K$ then $\operatorname{Lk}_{K}(\alpha)$ is also pure. 

The {\it dual graph} $\Lambda(K)$ of a pure simplicial poset $K$ is the graph whose vertices are the facets of $K$ and edges are the ordered pairs $(\{\sigma_1, \sigma_2\}, \gamma)$, where $\gamma$ is a ridge and is a common face of the facets $\sigma_1$, $\sigma_2$.  
Since each 1-cell of $K$ corresponds to a geometric $1$-simplex (closed ball of dimension one), it follows that the graph $\Lambda(K)$ has no loops. 
Observe, however, that $\Lambda(K)$ can have multiple edges running between the same two vertices. 

Suppose $(X, \leq)$ is the face poset of the simplicial CW-complex $M$. Let $M^{\prime}$ be the derived subdivision of $M$ (with the cells of $M$ identified with the vertices of $M'$). The face poset of $M^{\prime}$ is called the {\it derived subdivision} of $(X,\leq)$ and is denoted by $(X^{\prime}, \leq)$. The poset $X^{\prime}$ is (isomorphic to) the set of all chains of the form $\emptyset=\alpha_0< \alpha_1 < \cdots< \alpha_k$, for some $k\geq 0$, in $(X,\leq)$ and $C_1 \leq C_2$ if $C_1$ is a sub-chain of $C_2$. Since $M^{\prime}$ is a geometric simplicial complex, it follows that $(X^{\prime}, \leq)$ is a simplicial complex. In fact, if $\beta_1, \dots, \beta_m$ are common elements of two chains $C, D$ and $\beta_1< \cdots<\beta_m$ then $E := \emptyset < \beta_1< \cdots<\beta_m$ is a common sub-chain of $C, D$ and is in $X^{\prime}$. 

For a $d$-dimensional simplicial poset $K$, let $f_j(K)$ denote the number of $j$-cells of $K$ for $0\leq j\leq d$. The number $\chi(K) := \sum_{j=0}^{d} (-1)^jf_j(K)$ is called the {\em Euler characteristic of $K$}. 
The $(d+1)$-tuple $f(K) := (f_0(K), \dots, f_d(K))$ is called the {\it face vector} or {\em $f$-vector} of $K$.

\section{Suitability for a simplicial cell complex to allow quotients}\label{sec:quotient}

In this section we discuss the effect of taking the quotient of a simplicial cell complex by a group of automorphisms acting on it.

When considering a group $G$ acting on a simplicial complex $K$, the notions of {\em goodness} and {\em pureness} are known. Each one implies that $|K|/G \cong \left | K/G \right |$ (cf. \cite{BD12}). Where, the latter expression for the quotient is defined as the simplicial complex with vertex set the $G$-orbits of $V(K)$. More precisely, we have the following.

\begin{definition} \label{def:good_sc}
Let $G$ be a group of automorphisms of a
simplicial complex $K$. The action of $G$ on $K$ is called {\em good} if, for every pair $x\neq y$ of vertices of $K$ from a common $G$-orbit,
\begin{itemize}
  \item $xy$ is a non-edge; and 
  \item there exists a $g\in G$ such that $g(x) = y$, and $g$ fixes all the vertices of $\operatorname{Lk}(x)\cap \operatorname{Lk}(y)$.
\end{itemize}
\end{definition}

\begin{proposition}[Corollary 2.7 in \cite{BD12}] \label{prop:good}
Let $G$ be a group of automorphisms of a simplicial
complex $K$. If the action of $G$ on $K$ is good, then $|K/G| \cong |K|/G$.
\end{proposition}

To extend \Cref{prop:good} to simplicial cell complexes, we introduce the following definition.

\begin{definition} \label{def:good_scc}
Let $G \leq \operatorname{Aut}(X)$ be a group of automorphisms of a
simplicial cell complex $X$. We say that the action of $G$ on $X$ is {\em
good}, if two vertices of $X$ in the same $G$-orbit are not adjacent in $X$. 
\end{definition}

In the next lemma, we present an equivalent form of this customised definition of goodness (which we need in the proof of Lemma \ref{lem:good}). 

\begin{lemma} \label{lem:good_scc}
Let $G \leq \operatorname{Aut}(X)$ be a group of automorphisms of a
simplicial cell complex $X$. The following are equivalent. 
\begin{enumerate}[{$(i)$}]
\item 
For $\sigma \in X$ and $g \in G$, if $\alpha$ is a cell in $g(\sigma) \cap \sigma$, then $g(\alpha) = \alpha$. \newline $[$Observe that $g(\sigma)\in X$, but $\sigma\cap g(\sigma)$ need not be a single cell of $X$.$]$ 
\item The action of $G$ on $X$ is good. 
\end{enumerate}
\end{lemma}

\begin{proof}
First assume $(i)$ holds. If possible, suppose $(ii)$ does not hold. Then there exist two vertices $x\neq y$ in the same $G$-orbit which are adjacent via an edge $e$ of $X$. Since $x$ and $y$ are in the same $G$-orbit, there exists $g\in G$ such that  $g(x)=y$. Then $y$ is a 0-cell in $e \cap g(e)$. Therefore, by $(i)$, $g(y)=y$. Thus, $g(x) = y= g(y)$. This is not possible, since $g$ is an automorphism. Thus $(ii)$ holds. 

Now, suppose $(ii)$ holds. Let $\sigma\in X$ and $g\in G$. Suppose $v$ is a vertex in $\sigma\cap g(\sigma)$. Then $v$ is of $\sigma$ and $v= g(u)$ for some vertex $u$ of $\sigma$. If possible, suppose $u\neq v$. Then $u,v \in V(\sigma)$ and hence vertices of the edge $e= \sigma(\{u,v\})\leq \sigma$. Hence, $v=g(u)$ and $u$ are in the same $G$-orbit and vertices of the edge $e$. This is not possible since $(ii)$ holds. Therefore, $v=u =g^{-1}(v)$, and hence $g(v) =v$. Thus, $g(x)=x$ for each vertex $x$ in $\sigma\cap g(\sigma)$. 

Now suppose $\alpha$ is a cell in $\sigma\cap g(\sigma)$. Then $\alpha\leq \sigma$ and $\alpha=g(\beta)$ for some $\beta \leq \sigma$. Let $x$ be a vertex of $\alpha\leq \sigma$. Then $g^{-1}(x)$ is a vertex of $g^{-1}(\alpha)=\beta\leq \sigma$. This implies that $x$ is a vertex of $g(\sigma)$. Thus, $x$ is a vertex in $\sigma\cap g(\sigma)$. Therefore, by the above, $g(x) =x$.  Hence $x=g^{-1}(x)\in V(g^{-1}(\alpha)) = V(\beta)$. This implies that $V(\alpha) = V(\beta)$. Therefore, $\alpha = \sigma(V(\alpha)) = \sigma(V(\beta)) = \beta$. Thus, $(i)$ holds. 
\end{proof}

If $G\leq \operatorname{Aut}(Y)$ for a poset $(Y, \leq)$, then consider the set $Y/G$ of $G$-orbits of elements of $Y$. The partial order $\leq$ on $Y$ induces a canonical partial order on $Y/G$ (which we also denoted by $\leq$). Namely, we have $A\leq B$, for $A, B\in Y/G$, if there exist $a\in A$ and $b\in B$ such that $a\leq b$. Observe that the orbit $O_v$ of a vertex $v$ of $Y$ is a vertex of $Y/G$, the orbit $O_e$ of an 1-cell $e$ of $Y$ is a $1$-cell of $Y/G$ and the orbit $O_{\alpha}$ of an $i$-cell $\alpha$ of $Y$ is an $i$-cell of $Y/G$ for $0\leq i\leq \dim(Y)$. Thus, if $a_0 < \cdots < a_m$ is a chain in $Y$, then $O_{a_0} < \cdots < O_{a_m}$ is a chain in $Y/G$. This implies that $\dim(Y/G) =\dim(Y)$.  For example, if $X=\{\emptyset, a, b, e=[a, b]\}$, where $a< b\in \mathbb{R}$, then $X$ is a $1$-dimensional  CW complex. We have  $\mbox{Aut}(X) = \{\mbox{Id}, g\}\cong \mathbb{Z}_2$, where $g(a) =b$, $g(b)=a$ and $g(e)=e$. Here $X$ is a simplicial cell complex,  $X/\mbox{Aut}(X)$ consists of one vertex and one 1-cell and is a $1$-dimensional CW complex but not a simplicial cell complex. Thus, for a simplicial cell complex $(Y, \leq)$ and $G\leq \operatorname{Aut}(Y)$, $(Y/G, \leq)$ need not be a simplicial cell complex. Here we prove 

\begin{lemma} \label{lem:quotient}
Let $X$ be a simplicial cell complex, and let $G \leq \operatorname{Aut}(X)$. 
The quotient poset $X/G$ is a simplicial cell complex if and only if the action of $G$ on $X$ is good. 
\end{lemma}

\begin{proof}
First assume that the action of $G$ on $X$ is good. Since $X$ is a CW-complex, it follows that $X/G$ is a CW-complex as well. Let $\rho: X \to X/G$ be the canonical projection. Let $A\in X/G$. Then $A = \rho(\alpha)$ for some simplex $\alpha\in X$. If $a, b$ are two vertices of $\alpha$, then $a, b$ are vertices of the edge $e :=\alpha(\{a, b\}) \leq \alpha$. Since the action of $G$ on $X$ is good, it follows that orbits of $a$ and $b$ are distinct. This means $\rho(a) \neq \rho(b)$ for each pair of vertices $a, b$ of $\alpha$. Therefore, $A = \rho(\alpha)$ is a cell whose boundary is isomorphic to the boundary of $\alpha$. This implies that $X/G$ is a simplicial cell complex and the dimension of $X/G$ is the same as that of $X$. 

Conversely, suppose that the action of $G$ on $X$ is not good. Then there exist vertices $x, y$ of $X$ in the same $G$-orbit which are vertices of an edge $e$ of $X$. This implies that the 1-cell $\rho(e)$ is a loop in $X/G$. Therefore $X/G$ not regular and hence not a simplicial cell complex. This completes the proof. 
\end{proof}

For a simplicial cell complex $X$, let $X^{\prime}$ be its derived subdivision. So, if $\alpha \in V(X^{\prime})$, then $\alpha \in X$ and $V(\operatorname{Lk}_{X^{\prime}}(\alpha)) = \{\beta\in X \, : \, \beta < \alpha\} \sqcup \{\gamma\in X \, : \, \alpha\leq \gamma\}$. The group $\operatorname{Aut}(X)$ acts canonically on $X^{\prime}$ (namely, for $ \{\alpha_1 < \cdots < \alpha_k\}\in X^{\prime}$ and $g\in G$, $g(\{\alpha_1 < \cdots < \alpha_k\}) := \{g(\alpha_1) < \cdots < g(\alpha_k)\}$). Hence, we can assume that $\operatorname{Aut}(X) \leq \operatorname{Aut}(X^{\prime})$. In general, $\operatorname{Aut}(X) \neq\operatorname{Aut}(X^{\prime})$. For example, if $S^1_2$ is the simplicial cell complex with two vertices and two edges, then $\operatorname{Aut}(S^1_2) \cong \mathbb{Z}_2\times \mathbb{Z}_2$.  Whereas $(S^1_2)^{\prime}$ is a 4-cycle whose automorphism group is the dihedral group $D_8$ with eight elements.

\begin{lemma} \label{lem:good}
 Let $X$ be a simplicial cell complex with first derived subdivision $X^{\prime}$, and let $G \leq \operatorname{Aut}(X)$. If the action of $G$ on $X$ is good, then the action of $G$ on the simplicial complex $X^{\prime}$ is also good. 
\end{lemma}

\begin{proof}
Suppose two vertices $\alpha$ and $\beta$ of $X^{\prime}$ are in the same $G$-orbit. Let $\beta = g(\alpha)$, where $g\in G$. Then $\alpha$ and $\beta$ are cells of $X$ of the same dimension. Hence, $\alpha\not < \beta$ and $\beta\not < \alpha$ as cells of $X$. This implies that $\alpha\beta$ is a non-edge of $X^{\prime}$. 

\medskip

\noindent {\em Claim 1}: There does not exist $\gamma\in X$ such that $\alpha\leq \gamma$ and $\beta\leq \gamma$. 

\smallskip 

If possible, suppose there exists $\gamma\in X$ such that $\alpha\leq \gamma$, $\beta\leq \gamma$. Since a subset $A \subseteq V(\gamma)$ uniquely determines the face $\gamma(A)$ of $\gamma$ whose vertex set is $A$, it follows that $V(\alpha) \neq V(\beta)$. Since $g(\alpha) = \beta$, we have $g(V(\alpha)) = V(\beta)$ and hence $\#(V(\alpha)) = \#(V(\beta))$. These imply that there exists $x\in V(\alpha)\setminus V(\beta)$. Then $g(x)\in V(g(\alpha)) = V(\beta)$, and hence $g(x)\not = x$. Now, $x \leq \alpha \leq \gamma$ and $g(x) \leq \beta\leq \gamma$. Thus, $g(x)$ and $x$ are two distinct vertices of $\gamma$ and hence are vertices of the edge $\gamma(\{x, g(x)\})$. This is not possible since the action of $G$ on $X$ is good. This proves Claim 1. 

 \smallskip 
 
Now, suppose $\nu\in V(\textrm{lk}_{X^{\prime}}(\alpha)) \cap V(\textrm{lk}_{X^{\prime}}(\beta))$. By Claim 1, 
at least one of $\alpha$ and $ \beta$ is not a face of $\nu$.  Assume, without loss, that $\alpha\not \leq \nu$. 
Since $\nu \in V(\operatorname{Lk}_{X^{\prime}}(\alpha))$, it follows from the definition of the derived subdivision that $\nu < \alpha$ and hence $\#(V(\nu)) < \#(V(\alpha))$. 
Since $\alpha$ and $\beta$ are cells of the same dimension, we have  $\#(V(\alpha)) = \#(V(\beta))$ and hence $\#(V(\nu)) < \#(V(\beta))$. 
This implies that $\beta \not\leq \nu$. Again, since $\nu \in V(\operatorname{Lk}_{X^{\prime}}(\beta))$, it follows that $\nu < \beta$. Thus, $\nu$ is a cell  in $\alpha \cap \beta= \alpha \cap g(\alpha)$. Since the action of $G$ on $X$ is good, by Lemma \ref{lem:good_scc}, it follows that $g(\nu)=\nu$. Thus, $g$ fixes all the vertices in $\textrm{lk}_{X^{\prime}}(\alpha) \cap \textrm{lk}_{X^{\prime}}(\beta)$. Hence, the action of $G$ on the simplicial complex $X^{\prime}$ is good. 
\end{proof}

We have the following consequence of Lemma \ref{lem:good}. 

\begin{corollary} \label{cor:good}
 Let $X$ be a simplicial cell complex, and let $G \leq \operatorname{Aut}(X)$. 
 If the action of $G$ on $X$ is good, then $\left | X/G \right | \cong |X|/G$. 
\end{corollary}

\begin{proof}
Since $X$ is a simplicial cell complex, its first derived subdivision $X^{\prime}$ is a simplicial complex and hence $X^{\prime}/G$ is a simplicial complex. Since the action $G$ on $X$ is good, by Lemma \ref{lem:quotient}, $X/G$ is a simplicial cell complex and hence $(X/G)^{\prime}$ is a simplicial complex. 

Since $G \leq \operatorname{Aut}(X)$ we have $G \leq \operatorname{Aut}(X^{\prime})$. 
Observe that $V(X^{\prime}/G)$ is the set of vertices of $X^{\prime}/G$ and hence the set of $G$-orbits of vertices of $X^{\prime}$. This in particular implies that $V(X^{\prime}/G)$ is the set of $G$-orbits of cells of $X$, and thus equal to $V((X/G)^{\prime})$. Moreover, the identity map from $V(X^{\prime}/G)$ to $V((X/G)^{\prime})$ gives an isomorphism from $X^{\prime}/G$ to $(X/G)^{\prime}$ in the following way:
$A\in X^{\prime}/G$ $\Longrightarrow$ $A=\{O_{\alpha_1}, \dots, O_{\alpha_k}\}$ for some cells $\alpha_1, \dots, \alpha_k\in X$, where $\alpha_1< \cdots < \alpha_k$
$\Longrightarrow$ $O_{\alpha_1}, \dots, O_{\alpha_k}\in X/G= V((X/G)^{\prime})$ and $O_{\alpha_1}< \cdots < O_{\alpha_k}$
$\Longrightarrow$ $\{O_{\alpha_1}, \dots, O_{\alpha_k}\}\in (X/G)^{\prime}$ $\Longrightarrow$ $A\in (X/G)^{\prime}$. Similarly, $A\in (X/G)^{\prime}$ $\Longrightarrow$ $A\in X^{\prime}/G$. Thus, $X^{\prime}/G \cong (X/G)^{\prime}$ and hence $|X^{\prime}/G| \cong |(X/G)^{\prime}|$. 

Now, by Lemma \ref{lem:good}, the action of $G$ on $X^{\prime}$ is good. Hence, by Proposition \ref{prop:good}, $|X^{\prime}|/G \cong |X^{\prime}/G|$. 
Therefore, $|X|/G = |X^{\prime}|/G \cong |X^{\prime}/G| \cong |(X/G)^{\prime}| =|X/G|$. This completes the proof. 
\end{proof}

We end this section with the following beautiful example of good action of a group on a very well known simplicial cell complex. 

\begin{example}\label{exam:RP^n}
For $n\geq 1$, let $\rho : \mathbb{S}^n\to \mathbb{S}^n$ be the antipodal map. Then $\mathbb{S}^n/\langle\rho\rangle$ is the real projective space $\mathbb{RP}^n$. Let $K = S^0_2(\{u_0, v_0\})\ast\cdots\ast S^0_2(\{u_{n}, v_{n}\})$ be the join of $n+1$ copies of $S^0_2$. (This $K$ is known as the boundary of the $(n+1)$-dimensional cross polytope.) Then $K$ is a triangulation  of $\mathbb{S}^n$. The map $\rho$ induces an order two automorphism $\alpha$ of $K$ given by $\alpha|_{V(K)}=(u_0, v_0)\cdots(u_{n},v_{n})$. More explicitly,  for a simplex $x_{i_0}\cdots x_{i_k}\in K$ of dimension $k\leq n$, 
$\alpha(x_{i_0}\cdots x_{i_k}) = y_{i_0}\cdots y_{i_k}$, where $\{x_{i_j}, y_{i_j}\} = \{u_{i_j}, v_{i_j}\}$ for all $j$. Let $G :=\langle\alpha\rangle\leq \operatorname{Aut}(K)$. 
Then, in the category of simplicial complexes, the action of $G$ on the simplicial complex $K$ is not good and as a simplicial complex $K/G$ is (the closure of) an $n$-simplex.

In the category of simplicial cell complexes, however, the action of $G$ on the simplicial cell complex $K$ is good and $K/G$ is a  $(n+1)$-vertex simplicial cell structure of $\mathbb{RP}^n$. In particular, this produces a crystallization (simplicial cell structure with the smallest possible number of vertices) of $\mathbb{RP}^n$ for all $n > 0$.
\end{example}

\begin{remark}\label{rem:natural}
Lemma \ref{lem:quotient} shows that the good actions on the category of simplicial cell complexes are natural actions.
\end{remark}

\section{Simplicial cell decompositions of $(\mathbb{S}^2)^{n}$}\label{sec:scdecom}

In this section we present a very natural simplicial cell decomposition of the $n$-fold Cartesian power $(\mathbb{S}^2)^{n}$ of the $2$-sphere $\mathbb{S}^2$.

For this, consider the following well-known method to simplicially subdivide the $n$-fold Cartesian product of a triangle: Let the triangle be $\Delta=x_1x_2x_3$. Then the vertices of the cell $\Delta^n$ are $x_{i_1\cdots i_n} := (x_{i_1}, \dots, x_{i_n})$, $i_j\in \{1, 2, 3\}$. Take a subdivision of an $n$-cube $C^n= [1, 3]^n$ into $2^n$ sub-cubes. Let $C^{\prime}_n$ be the union of all these $2^n$ cubes, i.e., $C_n^{\prime} = ([1,2]\cup[2,3])^n$. The vertices of $C^{\prime}_n$ are $(i_1, i_2, \dots, i_n)$, where $i_j\in\{1, 2, 3\}$ for $1\leq j\leq n$. We identify $(i_1, i_2, \dots, i_n)$ with $x_{i_1i_2\cdots i_n}$.  A (maximal) monotone path from $x_{11\cdots 1}=(1, 1, \dots, 1)$ to $x_{33\cdots 3}=(3, 3, \dots, 3)$ in the edge graph of $C^{\prime}_n$ is a path of the form $x_{11\cdots1}\mbox{-}\cdots\mbox{-}x_{i_1i_2\cdots i_n}\mbox{-} x_{j_1j_2\cdots j_n}\mbox{-} \cdots\mbox{-} x_{33\cdots3}$, where $i_k\leq j_k$ for all $k$, and $j_1+\cdots + j_n=i_1+\cdots + i_n+1$. The facets of the simplicial subdivision of $\Delta^n$ are given by the vertex sets of maximal monotone paths. Such a path contains $2n$ edges and hence $2n+1$ vertices, and hence these facets are, indeed, of dimension $2n$ (see \Cref{fig:three} for an illustration of the case $n=2$). This type of subdivision is called {\em staircase subdivision}, and {\em path-simplex subdivision}  (see \cite[Page 382]{BCS88}, \cite[Page 67, Definition 8.8]{ES52} and \cite{LS17} for more details). To calculate the number of facets in this simplicial subdivision of $\Delta^n$, observe that, to reach $(3, \ldots , 3)$ from $(1 , \ldots , 1)$, we must follow every dimension of the cube exactly twice. There are ${2n \choose 2}$ possibilities to decide when to go along the first dimension, ${ 2n-2 \choose 2}$ options for the second, etc. It follows that we have $\prod \limits_{j=0}^{n-1} \binom{2(n-j)}{2}=\frac{(2n)!}{2^n}$ facets in the subdivision. 

\begin{figure}[htb]
\includegraphics[width=\textwidth]{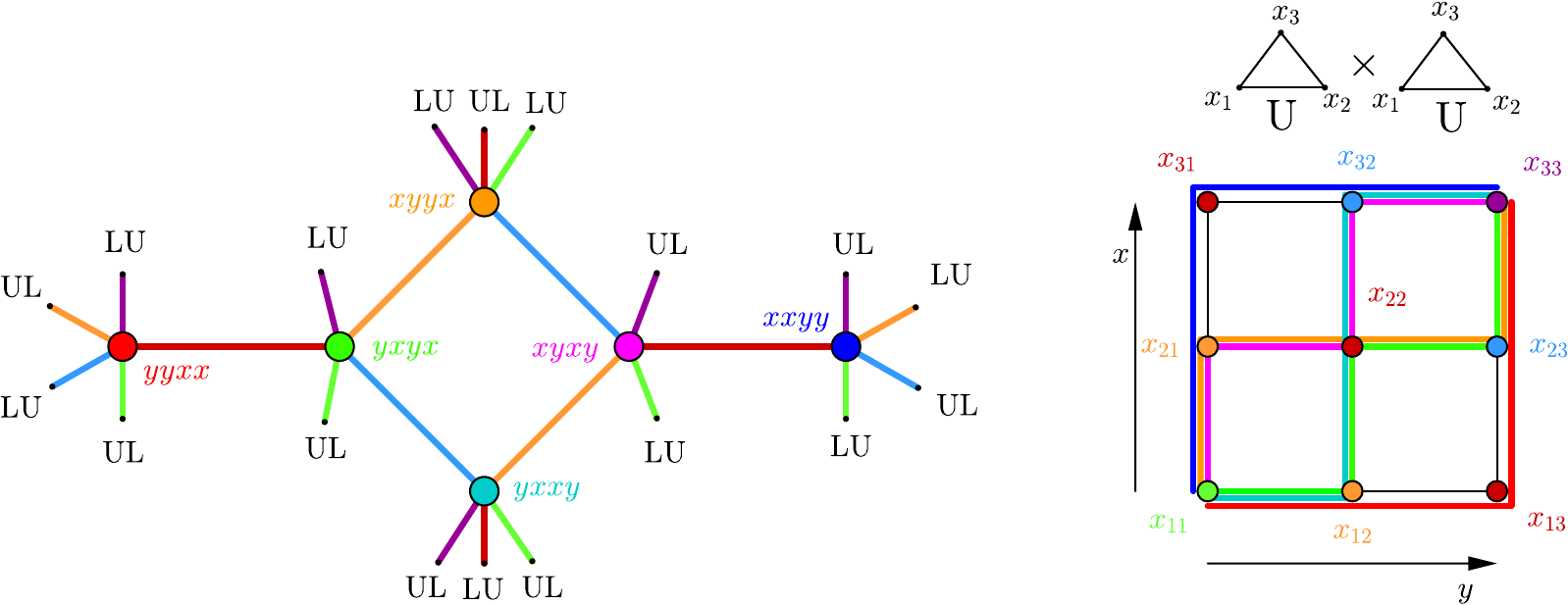}
\caption{The dual graph of the subdivision of $\Delta^2 \times \Delta^2$ into $6$ simplices of dimension $4$. Notice that this subdivision is a graph encoding.\label{fig:three}}
\end{figure}

\medskip

Let $S^2_3$ be the standard (unique) $3$-vertex $2$-triangle crystallisation of $\mathbb{S}^2$ (a pair of triangles, identified along their boundaries). We refer to the two triangles of $S^2_3$ as $U$ (for upper) and $L$ (for lower). Taking the $n$-fold topological Cartesian product of $S^2_3$ with itself, we obtain a cell decomposition $\widetilde{ X}^n$ of $(\mathbb{S}^2)^{n}$ with $2^n$ facets. Each facet is an $n$-fold Cartesian product of triangles and is uniquely described by a binary word whose letters are in $\{U,L\}$. We identify the facet $A_1\times\cdots\times A_n$ with the word $A_1\cdots A_n$, where $A_1, \dots, A_n \in \{U, L\}$. We have for the boundary of $\partial (A_1\cdots A_n) = \cup_{k=1}^n (A_1\times\cdots\times A_{k-1}\times (\partial A_k)\times A_{k+1} \times\cdots \times A_n)$. Thus, a facet has exactly $3n$ boundary ridges, each of them a Cartesian product of $n-1$ triangles and one edge. Moreover, the facets $A_1 \ldots A_n$ and $B_1 \cdots B_n$ share exactly three ridges if and only if $A_i = B_i$ for all but exactly one index $j = 1, \ldots , n$, and $A_j \neq B_j$. Otherwise they share none. These three ridges are exactly the ones determined by the $3$ edges of the $j$-th triangle. 

\medskip

We simplicially subdivide each cell of the complex $\widetilde{X}^n$ into ${(2n)!}/{2^n}$ facets, using the monotone path method given above. If carried out in all $2^n$ maximal cells, we obtain a total of $2^n \times ((2n)!/2^n) = (2n)!$ facets for the subdivision of $\widetilde{X}^n$. The subdivision is consistent on the ridges of $\widetilde{X}^n$ for the following reason: Let the vertices of $S^2_3$ be $x_1, x_2, x_3$. Moreover, let $F = A_1\cdots A_n$ be a maximal cell of $\widetilde{X}^n$, and let $F^{\prime}$ be its subdivision into simplices. Naturally, $F$ contains all $3^n$ vertices of $\widetilde{X}^n$. Each of its ridges $F_{\ell; i, j} = A_1 \times \cdots \times A_{\ell-1}\times x_ix_j\times A_{\ell+1}\times \cdots\times A_n$, $1\leq i<j\leq 3$ and $1\leq \ell\leq n$ is incident to exactly one other maximal cell $E = A_1 \cdots A_{\ell-1} B_{\ell} A_{\ell + 1} \cdots A_n$, where $\{A_{\ell}, B_{\ell}\} = \{U, L\}$.  By construction, we have $F_{\ell; i, j} = E_{\ell; i, j}$, and the restriction of the simplicial subdivision $F^{\prime}$ of $F$ on $F_{\ell; i, j} $ is equal to the restriction of the simplicial subdivision $E^{\prime}$ of $E$ on $F_{\ell; i, j}$. To see this, note that the $(2n-1)$-simplices in $F_{\ell; i, j}$ are the maximal monotone paths in the edge graph of the subdivision $C^n_{\ell;  i, j} = ([1, 2] \cup [2, 3])^{\ell - 1} \times [i, j] \times ([1, 2] \cup [2,3])^{n - \ell}$ of the $n$-dimensional rectangular parallelepiped (cube if $(i,  j) = (1, 3)$) $[1,3]^{\ell - 1} \times [i, j] \times [1, 3]^{n - \ell}$.

Altogether this yields the following.

\begin{theorem}\label{theo:scdX^n}
The simplicial cell complex
	\begin{align*}
X^n := &\{W_F \, : W \subseteq V(P), P \mbox{ is a maximal monotone path in the edge graph} \\
& \mbox {of } C^{\prime}_n, ~ V(P) \mbox{ is the vertex set of } P, \mbox{ and } F \mbox{ a maximal cell of } \widetilde{X}^n\} 
\end{align*}
is a simplicial cell decomposition of $({\mathbb S}^2)^n$. 
\end{theorem}

\begin{example}
\label{exam:Xn}
We have the following $f$-vectors of $X^n$, for $n\in \{1,2,3,4\}$:
\begin{itemize}
  \item $f(X^1) = (3,3,2)$,
  \item $f(X^2) = (9, 27, 58, 60, 24)$,
  \item $f(X^3) = (27, 189, 926, 2460, 3504, 2520, 720)$,
  \item $f(X^4) = (81, 1215, 12130, 64860, 194280, 337680, 338400, 181440, 40320)$.
\end{itemize}
In general, $X^n$ has $3^n$ vertices, $(2n)!$ facets, and Euler characteristic $2^n$. Moreover, $X^n$ can be represented as a graph encoding. That is, as an edge-coloured dual graph containing all the information to recover the simplicial cell complex. This can be followed from the construction, but we do not particularly stress this in this article.
See \Cref{app:code} for code to produce these simplicial cell complexes in low-dimensional topology software {\em Regina} \cite{regina}.
\end{example}

\section{Simplicial cell decompositions of $\mathbb{CP}^{\hspace{.3mm}n}$}\label{sec:CPn}

In this section we use the simplicial cell decomposition $X^n$ of $(\mathbb{S}^2)^n$ from \Cref{sec:scdecom} to construct a simplicial cell decomposition of $\mathbb{CP}^n$. For this, we first define an action of the group $\operatorname{Sym}(n)$ on $X^n$ realising coordinate permutation, and then show that this action is good. The quotient $X^n / \operatorname{Sym}(n)$ is our required simplicial cell complex.

The natural action of the symmetric group $\operatorname{Sym}(n)$ of degree $n$ by coordinate permutation on $(\mathbb{S}^{2})^n$ is also an action on $(S^2_3)^n$ by coordinate permutation. This gives an action of $\operatorname{Sym}(n)$ on $\widetilde{X}^n$: if $\pi\in \operatorname{Sym}(n)$ and $A = \alpha_1\times\cdots \times \alpha_n$ is a cell (not necessarily facet) in $\widetilde{X}^n$ then $\pi(\alpha_1\times\cdots \times \alpha_n)= \alpha_{\pi(1)}\times\cdots \times \alpha_{\pi(n)}$. 

Let $C_n =[1, 3]^n$ and $C_n^{\prime} = ([1,2]\cup[2,3])^n$ be as in Section \ref{sec:scdecom}. In particular, the vertices of $C^{\prime}_n$ are $x_{i_1i_2\cdots i_n} :=(i_1, i_2, \dots, i_n)$, $i_j\in\{1, 2, 3\}$, for $1\leq j\leq n$. 
Then, $\operatorname{Sym}(n)$ acts on the vertices of $C_n^{\prime}$ (which are also the vertices of $\widetilde{X}^n$ and $X^n$) as $\pi(i_1, \dots, i_n) = (\ell_1, \dots, \ell_n)$, where $\ell_k= i_{\pi(k)}$, $\pi\in\operatorname{Sym}(n)$. 

Here we prove 

\begin{lemma}\label{lem:path} 
The action of $\operatorname{Sym}(n)$ on $V(X^n)$ induces an action of $\operatorname{Sym}(n)$ on the simplicial cell complex $X^{n}$. 
\end{lemma}

\begin{proof}
Let $H$ be the edge graph of $C^{\prime}_n$. Then the distance $d_H(x_{i_1i_2\cdots i_n}, x_{11\cdots 1})$ between $x_{i_1i_2\cdots i_n}$ and $x_{11\cdots 1}$ in $H$ is $(i_1-1) + \cdots + (i_n-1) = (i_1+\cdots+i_n)-n$. 

\medskip 

\noindent {\em Claim} 2. If $x_{i_1i_2\cdots i_n}$ is a vertex of $C^{\prime}_n$ and $\pi\in \operatorname{Sym}(n)$ then $d_H(x_{i_1i_2\cdots i_n}, x_{11\cdots 1}) = d_H(\pi(x_{i_1i_2 \cdots i_n}), x_{11 \cdots 1})$.

\medskip 

If $\rho = (r, s)\in \operatorname{Sym}(n)$, $r<s$, is a transposition then $\rho(x_{i_1\cdots i_r\cdots i_s\cdots i_n}) = x_{i_1\cdots i_{r-1}i_s i_{r+1}\cdots i_{s-1}i_r i_{s+1}\cdots i_n}$. This implies $d_H(x_{i_1i_2 \cdots i_n}, x_{11\cdots 1}) = d_H(\rho(x_{i_1i_2\cdots i_n}),$ $x_{11\cdots 1})$. Since any permutation is the product of transpositions, this implies that $d_H(x_{i_1i_2\cdots i_n}, x_{11\cdots 1}) = d_H(\pi(x_{i_1i_2\cdots i_n}), x_{11\cdots 1})$. This proves Claim 2.

\medskip 

\noindent {\em Claim} 3. If $uv$ is an edge of $C^{\prime}_n$ and $\pi\in \operatorname{Sym}(n)$ then $\pi(u)\pi(v)$ is also an edge of $C^{\prime}_n$.

\medskip 

Suppose $u= x_{i_1i_2\cdots i_n}$ and $v= x_{j_1j_2\cdots j_n}$. Since $uv$ is an edge, there exists $\ell\in\{1, \dots, n\}$ such that $i_k = j_k$ for all $k\neq \ell$ and $|i_\ell -j_{\ell}| =1$.  Then $i_{\pi(k)} = j_{\pi(k)}$ for all $k\neq\ell$. Since $\sum_{r=1}^n i_r = \sum_{r=1}^n i_{\pi(r)}$ and $\sum_{r=1}^n j_r = \sum_{r=1}^n j_{\pi(r)}$, it follows that $|j_{\pi(\ell)} - i_{\pi(\ell)}|=1$. These imply $\pi(u)\pi(v) = x_{i_{\pi(1)}\cdots i_{\pi(n)}}x_{j_{\pi(1)}\cdots j_{\pi(n)}}$ is an edge of $C^{\prime}_n$. This proves Claim 3. 

\medskip

Let $P= v_0\mbox{-}v_1\mbox{-}\cdots\mbox{-}v_{2n-1}\mbox{-} v_{2n}$ be a maximal monotone path and $\pi\in \operatorname{Sym}(n)$. So, $v_0=x_{1\cdots 1} = \pi(x_{1\cdots 1}) = \pi(v_0)$ and $v_{2m}=x_{3\cdots 3} = \pi(x_{3\cdots 3}) = \pi(v_{2m})$. 
Then, by Claim 3, $\pi(v_{i-1})\pi(v_{i})$ is an edge of $C^{\prime}_n$ and by Claim 2, $d_H(\pi(v_i), \pi(v_0)) =i$ for $1\leq i\leq 2m$. This implies that 
$\pi(v_0)\mbox{-}\pi(v_1)\mbox{-}\cdots\mbox{-}\pi(v_{2n-1})\mbox{-} \pi(v_{2n})$ is a maximal monotone path in the edge graph of $C^{\prime}_n$. 
Therefore, for any maximal monotone path $P$ and $\pi\in \operatorname{Sym}(n)$, $\pi(V(P))$ is the vertex set of another maximal monotone path. Therefore, for any facet $V(P)_F$ of $X^n$, $\pi(V(P))_{\pi(F)}$ is a facet of $X^n$. These imply that $\pi(A_F) = \pi(A)_{\pi(F)}$, for $A_F\in X^n$ and $\pi\in \operatorname{Sym}(n)$, gives an action of $\operatorname{Sym}(n)$ on $X^n$. 
 This proves the lemma. 
 \end{proof}

The proof of \Cref{lem:path} implies the following statements. 

\begin{definition} \label{def:action-on-X^n}
The action of $\operatorname{Sym}(n)$ on $\widetilde{X}^n$ induces the following action on $X^n$. 
If $\{v_1, \dots, v_m\}_F$ is a simplex of $X^n$ and $\pi\in \operatorname{Sym}(n)$ then 
\begin{align*}\label{eq:action}
\pi(\{v_1, \dots, v_m\}_F) = \{\pi(v_1), \dots, \pi(v_m)\}_{\pi(F)}.
\end{align*}
\end{definition} 

We are now in a position to state and prove our main results.

\begin{theorem}\label{theo:good}
The action of \, $\operatorname{Sym}(n)$ on $X^n$ given in Definition \ref{def:action-on-X^n} is good. 
\end{theorem}

\begin{proof} {\em Claim} 4. Let 
$x_{11\cdots1}\mbox{-}\cdots\mbox{-}x_{i_1i_2\cdots i_n}\mbox{-} \cdots\mbox{-} x_{k_1k_2\cdots k_n}\mbox{-} \cdots\mbox{-} x_{33\cdots3}$ be a monotone path. 
Then there does not exist $\pi\in \operatorname{Sym}(n)$ such that $\pi(x_{i_1i_2\cdots i_n}) = x_{k_1k_2\cdots k_n}$. 

\medskip 

Since $x_{11\cdots1}\mbox{-}\cdots\mbox{-}x_{i_1i_2\cdots i_n}\mbox{-} \cdots\mbox{-} x_{k_1k_2 \cdots k_n}\mbox{-} \cdots\mbox{-} x_{33\cdots3}$ is a monotone path,  it follows that $d_H(x_{i_1i_2\cdots i_n}, x_{11\cdots 1}) < d_H(x_{k_1k_2\cdots k_n}, x_{11\cdots 1})$. Now, if $\pi\in \operatorname{Sym}(n)$ then, by Claim 2, $d_H(\pi(x_{i_1i_2\cdots i_n}), x_{11\cdots 1}) = d_H(x_{i_1i_2\cdots i_n}, x_{11\cdots 1}) < d_H(x_{k_1k_2\cdots k_n},$ $x_{11\cdots 1})$. Hence, $\pi(x_{i_1i_2\cdots i_n}) \neq x_{k_1k_2\cdots k_n}$. This proves Claim 4. 

\medskip

Suppose the action of $\operatorname{Sym}(n)$ on $X^n$ is not good. Then there exists a pair of vertices $x, y$ in a $\operatorname{Sym}(n)$-orbit which are vertices of an edge $e$. Since each edge is incident to a facet, we have a facet $\sigma$ such that $e\leq \sigma$. 
This implies, $x, y$ are vertices of $\sigma$ (in $F^{\prime}$ for some maximal cell $F$ of 
$\widetilde{X}^n$). Thus, $x, y$ are vertices of a maximal monotone path and are in a $\operatorname{Sym}(n)$-orbit. But this is not possible by Claim 4. This proves the result. 
\end{proof}

Now we have

\begin{theorem}\label{theo:cp^n}
The quotient $X^n/\operatorname{Sym}(n)$ is a simplicial cell complex and \linebreak $|X^n/\operatorname{Sym}(n)| \cong \mathbb{CP}^n$ for all $n\geq 2$.  
\end{theorem}

\begin{proof} 
By Theorem \ref{theo:good} and Lemma \ref{lem:quotient}, $X^n/\operatorname{Sym}(n)$ is a simplicial cell complex. This proves the first part. 

From Theorem \ref{theo:good} and Corollary \ref{cor:good}, we get $|X^n/\operatorname{Sym}(n)| \cong |X^n|/\operatorname{Sym}(n)$. Therefore, $|X^n/\operatorname{Sym}(n)| \cong |X^n|/\operatorname{Sym}(n) \cong |\widetilde{X}^n|/\operatorname{Sym}(n) \cong (\mathbb{S}^2)^n/\operatorname{Sym}(n)\cong \mathbb{CP}^n$. This proves the result. 
\end{proof}

Moreover, we have the following consequence of \Cref{theo:cp^n}. 

\begin{corollary}
The simplicial complex $(X^n)^{\prime}/\operatorname{Sym}(n)$ is a triangulation of $\mathbb{CP}^n$ for all $n\geq 2$. 
\end{corollary}

\begin{proof}
First note that $\operatorname{Sym}(n)\leq \operatorname{Aut}(X^n) \leq \operatorname{Aut}((X^n)^{\prime})$.
Now, from the proof of Corollary \ref{cor:good}, $(X^n)^{\prime}/\operatorname{Sym}(n)$ and $(X^n/\operatorname{Sym}(n))^{\prime}$ are isomorphic simplicial complexes. The result now follows from Theorem \ref{theo:cp^n}. 
\end{proof}

\begin{example}
We have the following $f$-vectors of $T_n = X^n/\operatorname{Sym}(n)$, for $n\in \{1,2,3,4\}$:
\begin{align}
  f(T_1) =& (\,3, 3, 2\,), \nonumber \\
  f(T_2) =& (\,6, 15, 30, 30, 12\,),   \nonumber \\
  f(T_3) =& (\,10, 46, 184, 440, 596, 420, 120\,), \nonumber \\
  f(T_4) =& ( \,15, 111, 764, 3\,345, 8\,982, 14\,700, 14\,280, 7\,560, 1\,680\,). \nonumber 
\end{align}
In particular, we have that $T_n$ has ${n+2 \choose 2}$ vertices, $(2n)!/n!$ facets, and Euler characteristic $n+1$ for these examples. As in the case of $X^n$ (see \Cref{exam:Xn}), $T_n$ can be represented as a graph encoding. That is, as an edge-coloured dual graph containing all the information to recover the simplicial cell complex. See \Cref{app:code} for code to produce these simplicial cell complexes in low-dimensional topology software {\em Regina} \cite{regina}.
\end{example}

\begin{corollary} \label{cor:f_2n}
Let $X^n$ and $T_n=X^n/{\rm Sym}(n)$ be the simplicial cell complexes as defined above. Then 
\begin{enumerate}[(i)]
  \item $f_{2n}(X^n) = (2n)!$, 
  \item $f_{2n}((X^n)^{\prime}) = (2n)!(2n+1)!$, 
  \item $f_{2n}(T_n) = (2n)!/n!$, and 
  \item $f_{2n}((T_n)^{\prime})$. 
\end{enumerate}
\end{corollary}

\begin{proof}
From the construction, we know that $f_{2n}(X^n) = (2n)!$. 

Since each $k$-simplex has $k+1$ faces of dimension $k-1$, it follows that  each $m$-simplex $\sigma$ of an $m$-dimensional simplicial (cell) complex $Y$ is in $(m+1)\times m\times (m-1) \times \cdots \times 2= (m+1)!$ maximal chains in $Y$. This implies that the number of maximal chains in $X^n$ is equal to $(2n)!\times (2n+1)!$. Hence the number of facets $f_{2n}((X^n)^{\prime})$ of $(X^n)^{\prime}$ is equal to $(2n)!(2n+1)!$.  This proves $(ii)$. 

Suppose $g(\alpha)=\alpha$ for some facet $\alpha$ of $X^n$ and $g\in {\rm Sym}(n)$.  Let $\beta$ be a neighbour of $\alpha$ in the dual graph $\Lambda(X^n)$. So,  $\gamma := \alpha\cap\beta$ is a ridge and a common face of $\alpha$ and $\beta$. So, $\gamma$ is a simplex in $\alpha\cap g(\alpha)=\alpha$. 
Since the action of ${\rm Sym}(n)$ on $X^n$ is good, by Lemma \ref{lem:good} (Lemma 3.4), $g(\gamma) = \gamma$. Then $g(\operatorname{Lk}_{X^n}(\gamma)) = \operatorname{Lk}_{X^n}(\gamma)$ and hence $g(\beta)=\beta$. Since the graph $\Lambda(X^n)$ is connected, this implies that $g(\sigma) = \sigma$ for any facet $\sigma$ of $X^n$. Thus $\sigma \equiv {\rm Id}$ on the set of facets of $X^n$. This implies that the action of ${\rm Sym}(n)$ on the set of facets of $X^n$ is free and hence $f_{2n}(T_n)= f_{2n}(X^n/{\rm Sym}(n)) = f_{2n}(X^n)/|{\rm Sym}(n)| = (2n)!/n!$. 

By similar arguments as in the proof of $(ii)$ above, $f_{2n}((T^n)^{\prime}) = f_{2n}(T^n) \times (2n+1)! = (2n)!(2n+1)!/n!$. This completes the proof. 
\end{proof}

\begin{example}
To compute the $f$-vectors of the simplicial trianglations of $\mathbb{C}P^n$, that is, 
$(T_n)'$, we have to count the number of faces in the derived subdivision of a triangulation
with a given $f$-vector. For the number of vertices this is simple: every face of the 
original triangulation contains one vertex of the subdivision in its centre, and hence
the number of vertices in the derived subdivision equals the sum of the face numbers of
the original complex. For the topdimensional faces (of, say, dimension $d$), every 
original facet is subdivided into $(d+1)!$ new facets.

To generalise this observation to all other face numbers, we claim that the set of 
$i$-dimensional faces in the interior of an original $j$-dimensional face $\Delta$ is in 
bijection with the set of surjective functions from $[j+1]=\{1, \ldots , j+1 \}$ to $[i+1]=\{ 1, \ldots , i+1\}$.
To see this, assume that the $j$-dimensional face has vertex labels $1, \ldots , j+1$.
For a given surjective function $f: [j+1] \to [i+1]$, the corresponding $i$-dimensional 
face has its first vertex in the barycentre of the face of $\Delta$ with vertex set $f^{-1} (1)$,
the second vertex in the barycentre of $f^{-1} ( \{1,2\})$, and so on. Every surjective function
will yield a unique such $i$-face, and every $i$-face uniquely determines a suitable surjective 
function. Take a moment to verify that this bijection also works to describe the number of 
facets, as well as the number of vertices of the subdivision.

The number of surjective functions $f: [j+1] \to [i+1]$, in turn, can be determined by
an application of the inclusion-exclusion principle as
\[ \sum \limits_{k=0}^{i} \left( (-1)^k \cdot {i+1 \choose k} \cdot (i+1-k)^{j+1} \right).\]
Note that, for $i=j$, this formula evaluates to $(j+1)!$.
In summary, given a complex with $f$-vector $(f_0, \ldots , f_d)$, the face vector
of its derived subdivision, $(f'_0, \ldots , f'_d)$ is given by
\[ f'_\ell = \sum \limits_{j=\ell}^{d} \sum \limits_{k=0}^{\ell} \left( (-1)^k \cdot {\ell+1 \choose k} \cdot (\ell+1-k)^{j+1} \right) \cdot f_j .\]

We have the following $f$-vectors for the simplicial triangulations of $\mathbb{C}P^n$, 
that is, $(T_n)'$, for $n\in \{1,2,3,4\}$:
\begin{align}
  f((T_1)') =& (\,8, 18, 12\,), \nonumber \\
  f((T_2)') =& (\,93, 990, 3\,060, 3\,600, 1\,440\,),   \nonumber \\
  f((T_3)') =& (\,1\,816, 66\,396, 549\,864, 1\,816\,800, 2\,843\,520, 2\,116\,800, 604\,800\,), \nonumber \\
  f((T_4)') =& ( \,51\,437, 5\,808\,816, 109\,509\,744, 767\,035\,800, 2\,621\,323\,440,  \nonumber \\
  & \,\,\, 4\,874\,990\,400, 5\,050\,684\,800, 2\,743\,372\,800, 609\,638\,400\,). \nonumber 
\end{align}
%
\end{example}

\begin{remark}
Let $X$ be a simplicial cell structure of a $d$-manifold and $g\in {\rm Aut}(X)$. Suppose $g(C) = C$ for some chain $C\equiv \emptyset< u < \alpha_1< \cdots < \alpha_d$. Then $g(u)=u$ and $g(\alpha_i)=\alpha_i$ for all $i$. We claim that $g\equiv {\rm Id}$. The claim is true for $d=1$. Assume the claim is true for dimension $<d$. Since $g(u)=u$, $g$ is an automorphism of $\operatorname{Lk}_X(u)$ and $g(\widetilde{C}) = \widetilde{C}$, where $\widetilde{C} \equiv u < \alpha_1< \cdots < \alpha_d$ is a chain in $\operatorname{Lk}_X(u)$. By inductive hypothesis, $g|_{\operatorname{Lk}_X(u)} \equiv {\rm Id}$. This implies that  $g(D) = d$ for all maximal chains $D$ containing $u$.  Let $V(\alpha_1)= \{u, v\}$. Since $g(\alpha_1)=\alpha_1$ and $g(u)=u$, it follows that $g(v)=v$. Thus $g(C^{\prime})=C^{\prime}$ for the chain $C^{\prime}\equiv \emptyset< v < \alpha_1< \cdots < \alpha_d$ and hence, by the above arguments, $g(E) =E$ for all maximal chains $E$ containing $v$. Since the edge graph of $X$ is connected, it follows that $g(A) =A$ for all maximal chains in $X$. Thus $g\equiv {\rm Id}$. This implies that  the action of ${\rm Aut}(X)$ (or any subgroup of ${\rm Aut}(X)$) is free on the set of maximal chains of $X$. In other words, ${\rm Aut}(X)$ is free on the set of facets of the derived complex $X^{\prime}$. Therefore, for $G\leq {\rm Sym}(n)$, the action of $G$ on the set of facets of $(X^n)^{\prime}$ is free  and hence $f_{2n}((X_n)^{\prime}/{\rm Sym}(n)) = f_{2n}((X_n)^{\prime})/|{\rm Sym}(n)|= (2n)!(2n+1)!/n!$. Thus, $f_{2n}((X^n/{\rm Sym}(n))^{\prime}) = f_{2n}((X_n)^{\prime}/{\rm Sym}(n)) = (2n)!(2n+1)!/n!$ and hence $f_{2n}(X^n/{\rm Sym}(n)) = (2n)!/n!$. These prove Corollary \ref{cor:f_2n} without using the goodness of the action of ${\rm Sym}(n)$. 
\end{remark}

\section{Implementation}\label{sec:code}

One advantage of our construction of simplicial cell decompositions of $(\mathbb{S}^2)^n$ and $\mathbb{CP}^n$ is that they can be explicitly implemented. In \Cref{app:code}, we provide python \cite{python} code to produce simplicial cell decompositions for $n \in \{2,3,4\}$ within low-dimensional topology software {\em Regina} \cite{regina}\footnote{In {\em Regina}, what we call simplicial cell decomposition is called a {\em triangulation}. We sometimes switch between the two terms whenever it is clear what we mean from context.}. This code is also available from \cite{DS2024}.
Since {\em Regina}'s python interface only implements simplicial cell decompositions up to real dimension eight, we added code to produce a graph encoding of our simplicial cell decompositions which works for arbitrary values of $n$.

Moreover, in \Cref{app:isosigs} we provide our simplicial cell structures of  $\mathbb{CP}^2$, $\mathbb{CP}^3$, and $\mathbb{CP}^4$ explicitly in the form of {\em Regina} (pseudo) triangulations or graph encodings (or both). 

\section*{Statements and Declarations}

Code and data corresponding to the family of simplicial cell decompositions described in this article are available from \Cref{app:code} and \Cref{app:isosigs}. The code to produce members from our family for larger $n$ are also available from public repository \cite{DS2024}.

The authors state that there is no conflict of interest.

{\small


\appendix

\section{Code to produce $X^n$ and $X^n/\operatorname{Sym}(n)$}
\label{app:code}

To run this code, first install low-dimensional topology software {\em Regina} \cite{regina}.
The code below runs in  {\em Regina}'s python interface \texttt{regina-python}. Alternatively.
similar code can also be run in python \cite{python}, after importing the {\em Regina} library; or
in the terminal environment from within the {\em Regina} graphical user interface.

See \cite{regina} for instructions on how to use {\em Regina} and extensive documentation.

\subsection{Helper functions to produce our simplicial cell decompositions}
\label{app:helper}

This code has been tested within the python interface of {\em Regina} 
\texttt{regina-python}. It is also available from \cite{DS2024}. The helper functions replicate our construction, described in detail in \Cref{sec:scdecom} and \Cref{sec:CPn}.

First, we need the help 
of the following two python packages.
\begin{lstlisting}[language=python]
# for permutation iterators
import itertools

# for computing the inverse of a permutation
from sympy.combinatorics import Permutation
\end{lstlisting}

Next we require a function to produce all path orderings defining the simplicial subdivision of
the $n$-fold cartesian product of the triangle, see the beginning of \Cref{sec:scdecom}.

\begin{lstlisting}[language=python]
def orderings(all_orderings,pos):
  # This function produces all monotone paths defining the simplices in
  # the canonical triangulation of the n-fold cartesian product of a triangle
  #
  # For a given n the function must be initiated with
  # all_orderings = list of length 2n filled with '-1'
  # pos = 1
  #
  # output is a list of lists of length 2n filled with entries 1, ... , n
  # length of list is (2n choose 2) x (2n-2 choose 2) x (4 choose 2)
  next_level = []
  # go through all partial paths found so far, and fill in the next value (pos)
  for lst in all_orderings:
    for i in range(len(lst)):
      # skip entries that are already filled
      if lst[i] != -1:
        continue
      # skip entries that are already filled
      for j in range(i+1,len(lst)):
        if lst[j] != -1:
          continue
        # fill every pair of unfilled entries with value 'pos'
        tmp = lst.copy()
        tmp[i] = pos
        tmp[j] = pos
        # add to new list of partial paths
        next_level.append(tmp)
  if not -1 in next_level[0]:
    # if done, return complete list of paths in cube
    return next_level
  else:
    # otherwise, call next level of recursion
    return orderings(next_level,3*pos)
\end{lstlisting}

With this, the simplicial subdivision of the $n$-fold product of the triangle is obtained 
by the following function (which produces an abstract simplicial complex)

\begin{lstlisting}[language=python]
def triangle_product(n):
  # This function produces all simplices in the canonical triangulation of 
  # the n-fold Cartesian product of the triangle with vertices 0, .... , 3^n-1
  #
  # number of facets is the same as number of monotone paths returned by
  # orderings(.,.) ((2n choose 2) x (2n-2 choose 2) x (4 choose 2) = 2n!/2^n)
  facets = []
  ords = orderings([[-1]*(2*n)],1)
  ords.sort()
  # turn paths into simplicial complex subdividing the n-fold Cartesian product of the triangle
  for ordering in ords:
    f=[0]
    for k in ordering:
      f.append(f[-1]+k)
    facets.append(f)
  return facets
\end{lstlisting}

The following function produces a list of boundary faces of the subdivided triangle product -- a simplicial complex.

\begin{lstlisting}[language=python]
def boundary_faces(facets):
  # This function returns the (Z2)-boundary of a simplicial complex defined by "facets"
  # record co-dim. one boundary faces of simplices as we see them
  faces = []
  # record their position as a pair [facet,boundary face index]
  pos = []
  for f in range(len(facets)):
    for k in range(len(facets[f])):
      tmp = facets[f].copy()
      tmp.remove(facets[f][k])
      if not tmp in faces:
        # add new boundary faces
        faces.append(tmp)
        pos.append([f,k])
      else:
        # remove boundary faces if we encounter them a second time
        pos.remove(pos[faces.index(tmp)])
        faces.remove(tmp)
  return faces,pos   
\end{lstlisting}

This simplicial complex needs to be converted into a {\em Regina} (pseudo) triangulation.

\begin{lstlisting}[language=python]
def simpToReg(simp,d,triangulation,perm):
  # Turn facet list of simplicial complex into Regina triangulation
  #
  # Here, this is only applied to the simplicial subdivision of the n-fold Cartesian
  # product of a triangle (2 x n = d)
  #
  # simp          = facet list (of abstract simplicial complex)
  # d             = dimension of triangulation (d = len(simp[0])-1)
  # triangulation = appropriate regina function for generating a d-dimensional triangulation
  # perm          = appropriate regina function for permuting elements of a simplex
  #                 (d+1) points
  facets = []
  for i in range(d+1):
    tmp = list(range(d+1))
    tmp.remove(i)
    facets.append(tmp)
  n = len(simp)
  neighbours = []
  gluings = []
  for j in range(n):
    neighbours.append([None] * (d+1))
    gluings.append([None] * (d+1))
    for f in range((d+1)):
      found = False
      for k in range(n):
        if k == j:
          continue
        good = True
        for i in range(d):
          if simp[j][facets[f][i]] not in simp[k]:
            good = False
            break
        if not good:
          continue
        found = True
        break
      if not found:
        continue
      neighbours[j][f] = k
      tmp = [None] * (d+1)
      unused = int(d*(d+1)/2)
      for i in range(d):
        adj = simp[k].index(simp[j][facets[f][i]])
        tmp[facets[f][i]] = adj
        unused = unused - adj
      tmp[f] = unused
      gluings[j][f] = perm(tmp)
  ans = triangulation()
  for j in range(n):
      ans.newSimplex()
  for j in range(n):
    for f in range(d+1):
      if ans.simplex(j).adjacentSimplex(f) == None and gluings[j][f] != None:
        ans.simplex(j).join(f,ans.simplex(neighbours[j][f]), gluings[j][f])
  return ans
\end{lstlisting}

Finally the following code produces the orbits of simplices of length $n!$ given by the 
$\operatorname{Sym}(n)$-action realising (complex) coordinate permutation

\begin{lstlisting}[language=python]
def orbits(n):
  # This function returns the set of orbits (each of length n!) of the natural
  # Sym(n)-action on the coordinates of the n-fold Cartesian product of
  # the 2-triangle 2-sphere
  # 
  # Idea is that an orbit is made out of all permutations of the letters of
  # the defining orbit (on the level of a simplex in a cell), with
  # target cell defined by permuting the letters of the cell label
  # 
  # For this, we use the following global counting of simplices:
  # 1. triangulation of one cell has ordering of simplices i according
  #    to lexicographic ordering of path labels (N simplices overall)
  # 2. For cell B_j = {U,L}^n, we can number from all U = 0 to all L = 2^n-1
  # 3. Global index of simplex is then idx = i + N x j
  # 
  # Orbits:
  #  (A) Every orbit must contain one representative of every permutation of paths
  #  (B) This gives groups of orbits of n! paths in {n choose #Us} = {n choose #Ls}
  #      facets
  #  (C) If a simplex is given by a path p in a cell c, then, for all sigma in
  #      Sym(n): produce sigma(p x c)
  # 
  # Coordinate axes are encoded as follows: 3^i ~ x_(i+1)
  lut = [1]
  for k in range(1,n):
    lut.append(3*lut[-1])
  lut.reverse()
  # list for orbits (each entry of length n!, (2n)!/n! orbits)
  orbs = []
  # paths defining simplices and number of simplices in triangulation of cell
  paths = orderings([[-1]*(2*n)],1)
  paths.sort()
  cell_len = len(paths) # this is equal to int(math.factorial(2*n)/2**n)
  # for conversion of cell number to binary
  frmt = "0"+str(n)+"b"
  # symmetric group on n elements
  all_perms = list(itertools.permutations(list(range(n))))
  # loop over all cells c
  for c in range(2**n):
    # loop over all simplices in cell c
    for q in range(cell_len):
      # check if we have seen this facet already
      already_done = False
      for o in orbs:
        if cell_len*c + q in o:
          already_done = True
          break
      if already_done:
        continue
      # found a representative for an orbit not yet processed
      # convert cell number to binary
      bnr = format(c, frmt)
      # collect group of orbits (orbit in cell) x (orbit of cells)
      all_bnrs = list(set(list(itertools.permutations(list(bnr)))))
      all_bnrs.sort()
      for seed in all_bnrs:
        cur_orb = []
        # permute paths by all permutation on n letters
        for p in range(len(all_perms)):
          # variable for permuted path
          new_path = []
          for k in paths[q]:
            new_path.append(lut[all_perms[p][lut.index(k)]])
          # determine target cell (permute cell)
          new_seed = list(seed)
          # take inverse of permutation
          sigma = Permutation(list(all_perms[p]))**(-1)
          # binary word for cell is in reverse with respect to coordinate axes enumeration
          new_bnr = []         
          new_seed.reverse()
          for k in list(sigma):
            new_bnr.append(new_seed[k])
          new_bnr.reverse()
          # represent cell label as binary string for conversion
          P = ''.join(new_bnr)
          dec = int(P,2)
          # add new facet in target cell to orbit
          cur_orb.append(dec*cell_len+paths.index(new_path))
        cur_orb.sort()
        orbs.append(cur_orb)
  orbs.sort()
  return orbs
\end{lstlisting}

\subsection{Code to produce sphere products}
\label{app:sp}

With the helper functions from \Cref{app:helper} in place, code to produce the
simplicial cell decompositions of $(\mathbb{S}^2)^n$, $n \in \{2,3,4\}$, is given below.

\begin{lstlisting}[language=python]
def sphere_product(NN,triangulation,perm):
  # return sphere product (works up to NN=4, regina only allows triangulations
  # up to 8 = 2 x 4 real dimensions)
  # 
  # NN            = (complex) dimension of complex projective space
  # triangulation = appropriate regina function for generating a d-dimensional triangulation
  # perm          = appropriate regina function for permuting elements of a simplex
  #                 (d+1) points
  O = orderings([[-1]*(2*NN)],1)
  O.sort()
  F = triangle_product(NN)
  faces,pos = boundary_faces(F)
  len_faces = len(F)
  Tri = simpToReg(F,2*NN,triangulation,perm)
  # for conversion to binary
  frmt = "0"+str(NN)+"b"
  # build up triangulation from disjoint triangulations of cells
  # cells are numbered by which number they represent in binary - 1
  T = triangulation()
  for c in range(2**NN):
    T.insertTriangulation(Tri)
  # loop over cells
  for c in range(2**NN):    
    # make buckets of free faces to connect to neighbour cells:
    buckets = []
    for i in range(NN):
      buckets.append([])
    for j in range(len(faces)):
      for char in range(NN-1):
        # Algorithm:
        # char=0: len=2: first
        # char=0, len=3: not first (done for n=2)
        #  char=1, len=4: second
        #  char=1, len=5: not first not second (done for n=3)
        #    char=2, len=6: third
        #    char=2, len=7: not first, second, third (done for n=4)
        #      ...
        # for entry 2*(char+1)
        ctr1 = 0
        # for entry 2*(char+1)+1
        ctr2 = 0
        res=[]
        for k in faces[j]:
          if not k%3**(char+1) in res:
            res.append(k%3**(char+1))
        if len(res) == 2*(char+1):
          buckets[char].append(pos[j])
          break
        if len(res) == 2*(char+1)+1 and char == NN-2:
          buckets[char+1].append(pos[j])
    # decide to which cell to glue to
    bnr = format(c, frmt)
    for char in range(NN):
      # already glued
      if bnr[char] == '1':
        continue
      ll = list(bnr)
      ll[char] = '1'
      newBnr=''.join(ll)
      newC = int(newBnr,2)
      for buc in buckets[char]:
        T.simplex(len_faces*c+buc[0]).join(buc[1],T.simplex(buc[0]+len_faces*newC),perm())
  return T
\end{lstlisting}

The sphere products in dimensions $4$, $6$, and $8$ can then be called as follows.

\begin{lstlisting}[language=python]
def s2xs2():
  # returns a triangulation of S2 x S2
  return sphere_product(2,Triangulation4,Perm5)
  
def s2xs2xs2():
  # returns a triangulation of S2 x S2 x S2
  return sphere_product(3,Triangulation6,Perm7)
  
def s2xs2xs2xs2():
  # returns a triangulation of S2 x S2 x S2 x S2
  return sphere_product(4,Triangulation8,Perm9)
\end{lstlisting}

\subsection{Code to produce $\mathbb{CP}^n$}
\label{app:cpn}

Finally, we can produce the code to construct simplicial cell decompositions 
$T_n = X^n/\operatorname{Sym}(n)$, for $n\in \{2,3,4\}$:

\begin{lstlisting}[language=python]
def CP(NN,triangulation,perm):
  # calculate orbits
  orbs = orbits(NN)
  # compute sphere product
  sp = sphere_product(NN,triangulation,perm)
  # empty triangulation to contain complex projective space
  T = triangulation()
  # compute paths for simplicial subdivision of single cell
  paths = orderings([[-1]*(2*NN)],1)
  # sort (IMPORTANT to do this after every call of 'orderings(.,.)')
  paths.sort()
  # go through orbits and add one simplex per orbit
  for o in range(len(orbs)):
    T.newSimplex()
  # for every smallest orbit representative: detect gluing from sphere product sp
  for k in range(len(orbs)):
    # go through all orbit reps and check their adjacent orbits
    for simp in orbs[k]:
      for l in range(2*NN+1):
        adj = sp.simplex(simp).adjacentSimplex(l)
        target = False
        for o in range(len(orbs)):
          if adj.index() in orbs[o]:
            target = o
            break
    rep = orbs[k][0]
    for l in range(2*NN+1):
      adj = sp.simplex(rep).adjacentSimplex(l)
      # if face is still free
      if T.simplex(k).adjacentSimplex(l) == None:
        # get target simplex
        for o in range(len(orbs)):
          if adj.index() in orbs[o]:
            target = o
            break
        if T.simplex(target).adjacentSimplex(l) == None:
          T.simplex(k).join(l,T.simplex(target),perm())
  return T
\end{lstlisting}

\noindent
The complex spaces in (real) dimensions $4$, $6$, and $8$ can then be called as follows.

\begin{lstlisting}[language=python]
def CP2():
  # returns a triangulation of CP2
  return CP(2,Triangulation4,Perm5)
  
def CP3():
  # returns a triangulation of CP3
  return CP(3,Triangulation6,Perm7)
  
def CP4():
  # returns a triangulation of CP4
  return CP(4,Triangulation8,Perm9)
\end{lstlisting}

\section{Explicit presentation of our simplicial cell complexes}
\label{app:isosigs}

In {\em Regina} \cite{regina}, triangulations are represented by {\em isomorphism signatures}.
An isomorphism signature of a triangulation is an encoding of the face gluings of a triangulation
into a string such that two triangulations are isomorphic if and only if their isomorphism
signatures are equal. See \cite{Bur11c} for more details on how these signatures are defined. For
our purposes it is enough to know that each such string contains enough information to 
uniquely reconstruct a triangulation.

In this section we provide isomorphism signatures of the (pseudo) triangulations of 
$(\mathbb{S}^2)^n$, and $\mathbb{CP}^n$, $n \in \{1,2,3,4\}$ described in this article,
which can be used to directly construct the (pseudo) triangulations using {\em Regina} 
\cite{regina} command 
\[\texttt{T = TriangulationN.fromIsoSig()}\] 
(where \texttt{N} is to be replaced by the appropriate real dimension).

Since isomorphism signatures require a working installation of {\em Regina} to be useful,
we also present all triangulations as graph encodings. Graph encodings are edge-coloured
graphs encoding simplicial cell complexes, see \cite{FGG1986} for details. Every node of 
the graph represents a facet, and every arc represents a ridge. The colour of the arc 
encodes which faces are glue together. Given the edges of a graph encoding, a small script
can turn it into a Regina triangulation and vice versa.

\subsection{$\mathbb{S}^2 \times \mathbb{S}^2$ and $\mathbb{CP}^2$}

Isomorphism signature of $\mathbb{S}^2 \times \mathbb{S}^2$:

\begin{lstlisting}[language=python]
yLwMMLvAzQQvQAQAPzQMQbceeefgjjkkmnonlpqrqppsqsruuuwvwvxxxxaaaaaaaaaaaaaaaaaaaaaaaaaaaaaaa
aaaaaaaaaaaaaaaaaaaaaaaaaaaaaaaaaaaaaaaaaaa
\end{lstlisting}

\noindent
Triangulation of $\mathbb{S}^2 \times \mathbb{S}^2$ as a graph encoding:

\begin{lstlisting}[language=python]
Colour 0:
((0,12),(1,13),(2,14),(3,9),(4,10),(5,11),(6,18),(7,19),(8,20),(15,21),(16,22),(17,23))

Colour 1:
((0,12),(1,3),(2,4),(5,11),(6,18),(7,9),(8,10),(13,15),(14,16),(17,23),(19,21),(20,22))

Colour 2:
((0,1),(2,8),(3,15),(4,5),(6,7),(9,21),(10,11),(12,13),(14,20),(16,17),(18,19),(22,23))

Colour 3:
((0,6),(1,2),(3,4),(5,17),(7,8),(9,10),(11,23),(12,18),(13,14),(15,16),(19,20),(21,22))

Colour 4:
((0,6),(1,7),(2,14),(3,9),(4,16),(5,17),(8,20),(10,22),(11,23),(12,18),(13,19),(15,21))
\end{lstlisting}

\noindent
Isomorphism signature of $\mathbb{CP}^2$:

\begin{lstlisting}[language=python]
mLwMMLLAQQQbceeefgkjjlkjklljlkaaaaaaaaaaaaaaaaaaaaaaaaaaaaaaaaaaaaaa
\end{lstlisting}

\noindent
Simplicial cell decomposition of $\mathbb{CP}^2$ as a graph encoding:

\begin{lstlisting}[language=python]
Colour 0: 
((0,8),(1,7),(2,6),(3,9),(4,10),(5,11))

Colour 1: 
((0,8),(1,2),(3,9),(4,6),(5,7),(10,11))

Colour 2: 
((0,1),(2,5),(3,4),(6,11),(7,8),(9,10))

Colour 3: 
((0,3),(1,2),(4,5),(6,7),(8,9),(10,11))

Colour 4: 
((0,3),(1,4),(2,6),(5,11),(7,10),(8,9))
\end{lstlisting}

\subsection{$\mathbb{S}^2 \times \mathbb{S}^2 \times \mathbb{S}^2$ and $\mathbb{CP}^3$}

The isomorphism signature of $(S^2)^3$ is somewhat too long to present it here. It can 
be generated using the code in \cite{DS2024}. The isomorphism signature of $\mathbb{CP}^3$:

\begin{lstlisting}[language=python]
-c4bLLMLzLPLQwvALAPAvMQQQwQMQAvvPwAALMQMAQPAQAPQQQQAPQPwzzAQLwQMAQQPAPAQQMQLQMPQLQQMQAvwA
MQQQAAAAQLQwQQPQPPQQQQQPAvMAQAMAQQQQPQPQAQQwPMQQPQQQAQQcaeadafahafajajajahakalaoanapapaqa
ravasasawaxaqazawaBaraDaCawauaAaAatayaEaBauaGaFazaCaGaGaLaJaNaKaOaHaHaSaIaVaTaMaJaVaNaKaK
aOaPaWaQaQaUaXaZaRaYaYaHaQa0aIaYaWa1aRaWaVaLaRaKaKaIa0aSaXaZa1aTa3a3aYaVa5a6a9a9a+a-a-a+a
8a-a-acbcbabdbdbab9acb8a7aebgbdb8a5afbbb6aibcb7a7afbbbib8adb9aebfbmb8agbbbmb5akbebgb6akbo
bhbhbqblbmbibqbrbnbpbjbrbnbtbubvbtbvbzbxbzbxbxbubzbxbzbvbybxbxbzbubBbubybvbBbsbFbDbybFbAb
wbGbAbBbtbHbIbCbDbDbybJbKbEbAbKbJbsbsbIbCbGbEbwbHbEbJbFbFbKbGbLbKbNbObPbPbQbRbQbRbMbPbRbS
bPbTbUbObObUbNbObMbSbRbTbMbSbVbSbTbQbTbUbNbNbWbMbMbSbXbWbXbVbVbXbZbZbZb1b0b1b1b1bYb1bZbZb
2b0bYbYb2b0b0b2b2bYbYb3b3b3b3b3b3baaaaaaaaaaaaaaaaaaaaaaaaaaaaaaaaaaaaaaaaaaaaaaaaaaaaaaa
aaaaaaaaaaaaaaaaaaaaaaaaaaaaaaaaaaaaaaaaaaaaaaaaaaaaaaaaaaaaaaaaaaaaaaaaaaaaaaaaaaaaaaaaa
aaaaaaaaaaaaaaaaaaaaaaaaaaaaaaaaaaaaaaaaaaaaaaaaaaaaaaaaaaaaaaaaaaaaaaaaaaaaaaaaaaaaaaaaa
aaaaaaaaaaaaaaaaaaaaaaaaaaaaaaaaaaaaaaaaaaaaaaaaaaaaaaaaaaaaaaaaaaaaaaaaaaaaaaaaaaaaaaaaa
aaaaaaaaaaaaaaaaaaaaaaaaaaaaaaaaaaaaaaaaaaaaaaaaaaaaaaaaaaaaaaaaaaaaaaaaaaaaaaaaaaaaaaaaa
aaaaaaaaaaaaaaaaaaaaaaaaaaaaaaaaaaaaaaaaaaaaaaaaaaaaaaaaaaaaaaaaaaaaaaaaaaaaaaaaaaaaaaaaa
aaaaaaaaaaaaaaaaaaaaaaaaaaaaaaaaaaaaaaaaaaaaaaaaaaaaaaaaaaaaaaaaaaaaaaaaaaaaaaaaaaaaaaaaa
aaaaaaaaaaaaaaaaaaaaaaaaaaaaaaaaaaaaaaaaaaaaaaaaaaaaaaaaaaaaaaaaaaaaaaaaaaaaaaaaaaaaaaaaa
aaaaaaaaaaaaaaaaaaaaaaaaaaaaaaaaaaaaaaaaaaaaaaaaaaaaaaaaaaaaaaaaaaaaaaaaaaaaaaaaaaaaaaaaa
aaaaaaaaaaaaaaaaaaaaaaaaaaaaaaaaaaaaaaaaaaaaaaaaaaaaaaaaaaaaaaaaaaaaaaaaaaaaaaaaaaaaaaaaa
aaaaaaaaaaaaaaaaaaaaaaaaaaaaaaaaaaaaaaaaaaaaaaa
\end{lstlisting}

The simplicial cell decomposition of $\mathbb{CP}^3$ represented as a graph encoding has two automorphisms
of order $2$. The first leaves colours invariant, the second permutes colour $i$ to 
$6-i$, $0 \leq i \leq 6$.

The automorphisms are given as

\begin{lstlisting}[language=python]
(1,106)(2,107)(3,108)(4,109)(5,110)(6,111)(7,112)(8,113)(9,114)(10,115)(11,116)(12,117)
(13,118)(14,119)(15,120)(16,93)(17,92)(18,91)(19,90)(20,89)(21,88)(22,102)(23,101)(24,100)
(25,105)(26,104)(27,103)(28,97)(29,96)(30,99)(31,98)(32,94)(33,95)(34,81)(35,80)(36,79)
(37,85)(38,84)(39,87)(40,86)(41,82)(42,83)(43,76)(44,77)(45,78)(46,69)(47,68)(48,67)
(49,73)(50,72)(51,75)(52,74)(53,70)(54,71)(55,64)(56,65)(57,66)(58,61)(59,62)(60,63)
\end{lstlisting}

\noindent
for the colour preserving one, and 

\begin{lstlisting}[language=python]
(2,4)(3,7)(6,8)(11,12)(16,58)(17,55)(18,43)(19,48)(20,36)(22,59)(23,56)(24,44)(25,60)
(26,57)(27,45)(28,53)(29,41)(30,54)(31,42)(34,47)(37,50)(39,52)(61,93)(62,102)(63,105)
(64,92)(65,101)(66,104)(67,90)(68,81)(70,97)(71,99)(72,85)(74,87)(76,91)(77,100)(78,103)
(79,89)(82,96)(83,98)(107,109)(108,112)(111,113)(116,117)
\end{lstlisting}

\noindent
for the one shuffling colours.
Orbit representatives are then given for the first four colours $0$ to $3$ as follows:

\begin{lstlisting}[language=python]
Colour 0:
(1,58),(2,59),(3,60),(4,55),(5,56),(6,57),(7,48),(8,47),(9,46),(10,53),(11,54),(12,50),
(13,49),(14,52),(15,51),(16,88),(17,89),(18,90),(22,79),(23,80),(24,81),(25,76),(26,77),
(27,78),(28,84),(29,85),(30,82),(31,83),(32,87),(33,86)

Colour 1:
(1,58),(2,59),(3,60),(4,7),(5,8),(6,9),(10,12),(11,13),(14,15),(16,88),(17,89),(18,90),
(22,25),(23,26),(24,27),(28,30),(29,31),(32,33),(34,46),(35,47),(36,48),(37,49),(38,50),
(39,51),(40,52),(41,53),(42,54),(43,55),(44,56),(45,57)

Colour 2:
(1,4),(2,5),(3,6),(7,43),(8,44),(9,45),(10,11),(12,14),(13,15),(16,22),(17,23),(18,24),
(19,34),(20,35),(21,36),(25,67),(26,68),(27,69),(28,37),(29,38),(30,39),(31,40),(32,41),
(33,42),(49,51),(50,52),(53,54),(55,58)(56,59),(57,60)

Colour 3:
(1,21),(2,3),(4,7),(5,10),(6,11),(8,12),(9,13),(14,32),(15,33),(16,61),(17,19),(18,20),
(22,25),(23,28),(24,29),(26,30),(27,31),(34,37),(35,38),(36,43),(39,74),(40,75),(41,44),
(42,45),(46,49),(47,50),(48,55),(53,56),(54,57),(59,60),
\end{lstlisting}

\subsection{$\mathbb{CP}^4$}

The simplicial cell decomposition of $(S^2)^4$ has $8! = 40320$ simplices of dimension eight, and is, hence, too large to be presented here. It can be generated using the code in \cite{DS2024} either as a Regina \cite{regina} (pseudo) triangulation or as a graph encoding.

The same goes for the simplicial cell decomposition of $\mathbb{CP}^4$, which has $1680$ simplices. We choose the representation
as a graph encoding. Moreover, we use the fact that the simplicial cell decomposition has two automorphisms
of order $2$. The first leaves colours invariant, the second permutes colour $i$ to 
$8-i$, $0 \leq i \leq 8$.

The automorphisms are given as

\begin{lstlisting}[language=python]
(   1,1576)(   2,1577)(   3,1578)(   4,1579)(   5,1580)(   6,1581)(   7,1582)(   8,1583)
(   9,1584)(  10,1585)(  11,1586)(  12,1587)(  13,1588)(  14,1589)(  15,1590)(  16,1591)
(  17,1592)(  18,1593)(  19,1594)(  20,1595)(  21,1596)(  22,1597)(  23,1598)(  24,1599)
(  25,1600)(  26,1601)(  27,1602)(  28,1603)(  29,1604)(  30,1605)(  31,1606)(  32,1607)
(  33,1608)(  34,1609)(  35,1610)(  36,1611)(  37,1612)(  38,1613)(  39,1614)(  40,1615)
(  41,1616)(  42,1617)(  43,1618)(  44,1619)(  45,1620)(  46,1621)(  47,1622)(  48,1623)
(  49,1624)(  50,1625)(  51,1626)(  52,1627)(  53,1628)(  54,1629)(  55,1630)(  56,1631)
(  57,1632)(  58,1633)(  59,1634)(  60,1635)(  61,1636)(  62,1637)(  63,1638)(  64,1639)
(  65,1640)(  66,1641)(  67,1642)(  68,1643)(  69,1644)(  70,1645)(  71,1646)(  72,1647)
(  73,1648)(  74,1649)(  75,1650)(  76,1651)(  77,1652)(  78,1653)(  79,1654)(  80,1655)
(  81,1656)(  82,1657)(  83,1658)(  84,1659)(  85,1660)(  86,1661)(  87,1662)(  88,1663)
(  89,1664)(  90,1665)(  91,1666)(  92,1667)(  93,1668)(  94,1669)(  95,1670)(  96,1671)
(  97,1672)(  98,1673)(  99,1674)( 100,1675)( 101,1676)( 102,1677)( 103,1678)( 104,1679)
( 105,1680)( 106,1383)( 107,1382)( 108,1381)( 109,1380)( 110,1379)( 111,1378)( 112,1392)
( 113,1391)( 114,1390)( 115,1395)( 116,1394)( 117,1393)( 118,1387)( 119,1386)( 120,1389)
( 121,1388)( 122,1384)( 123,1385)( 124,1371)( 125,1370)( 126,1369)( 127,1375)( 128,1374)
( 129,1377)( 130,1376)( 131,1372)( 132,1373)( 133,1366)( 134,1367)( 135,1368)( 136,1359)
( 137,1358)( 138,1357)( 139,1363)( 140,1362)( 141,1365)( 142,1364)( 143,1360)( 144,1361)
( 145,1354)( 146,1355)( 147,1356)( 148,1351)( 149,1352)( 150,1353)( 151,1473)( 152,1472)
( 153,1471)( 154,1470)( 155,1469)( 156,1468)( 157,1482)( 158,1481)( 159,1480)( 160,1485)
( 161,1484)( 162,1483)( 163,1477)( 164,1476)( 165,1479)( 166,1478)( 167,1474)( 168,1475)
( 169,1461)( 170,1460)( 171,1459)( 172,1465)( 173,1464)( 174,1467)( 175,1466)( 176,1462)
( 177,1463)( 178,1456)( 179,1457)( 180,1458)( 181,1503)( 182,1502)( 183,1501)( 184,1500)
( 185,1499)( 186,1498)( 187,1512)( 188,1511)( 189,1510)( 190,1515)( 191,1514)( 192,1513)
( 193,1507)( 194,1506)( 195,1509)( 196,1508)( 197,1504)( 198,1505)( 199,1491)( 200,1490)
( 201,1489)( 202,1495)( 203,1494)( 204,1497)( 205,1496)( 206,1492)( 207,1493)( 208,1486)
( 209,1487)( 210,1488)( 211,1548)( 212,1547)( 213,1546)( 214,1551)( 215,1550)( 216,1549)
( 217,1543)( 218,1542)( 219,1545)( 220,1544)( 221,1540)( 222,1541)( 223,1560)( 224,1559)
( 225,1558)( 226,1563)( 227,1562)( 228,1561)( 229,1555)( 230,1554)( 231,1557)( 232,1556)
( 233,1552)( 234,1553)( 235,1572)( 236,1571)( 237,1570)( 238,1575)( 239,1574)( 240,1573)
( 241,1567)( 242,1566)( 243,1569)( 244,1568)( 245,1564)( 246,1565)( 247,1525)( 248,1524)
( 249,1527)( 250,1526)( 251,1522)( 252,1523)( 253,1531)( 254,1530)( 255,1533)( 256,1532)
( 257,1528)( 258,1529)( 259,1537)( 260,1536)( 261,1539)( 262,1538)( 263,1534)( 264,1535)
( 265,1516)( 266,1517)( 267,1518)( 268,1519)( 269,1520)( 270,1521)( 271,1413)( 272,1412)
( 273,1411)( 274,1417)( 275,1416)( 276,1419)( 277,1418)( 278,1414)( 279,1415)( 280,1408)
( 281,1409)( 282,1410)( 283,1425)( 284,1424)( 285,1423)( 286,1429)( 287,1428)( 288,1431)
( 289,1430)( 290,1426)( 291,1427)( 292,1420)( 293,1421)( 294,1422)( 295,1441)( 296,1440)
( 297,1443)( 298,1442)( 299,1438)( 300,1439)( 301,1447)( 302,1446)( 303,1449)( 304,1448)
( 305,1444)( 306,1445)( 307,1453)( 308,1452)( 309,1455)( 310,1454)( 311,1450)( 312,1451)
( 313,1432)( 314,1433)( 315,1434)( 316,1435)( 317,1436)( 318,1437)( 319,1396)( 320,1397)
( 321,1398)( 322,1399)( 323,1400)( 324,1401)( 325,1402)( 326,1403)( 327,1404)( 328,1405)
( 329,1406)( 330,1407)( 331,1284)( 332,1283)( 333,1282)( 334,1288)( 335,1287)( 336,1290)
( 337,1289)( 338,1285)( 339,1286)( 340,1279)( 341,1280)( 342,1281)( 343,1276)( 344,1277)
( 345,1278)( 346,1308)( 347,1307)( 348,1306)( 349,1312)( 350,1311)( 351,1314)( 352,1313)
( 353,1309)( 354,1310)( 355,1303)( 356,1304)( 357,1305)( 358,1320)( 359,1319)( 360,1318)
( 361,1324)( 362,1323)( 363,1326)( 364,1325)( 365,1321)( 366,1322)( 367,1315)( 368,1316)
( 369,1317)( 370,1336)( 371,1335)( 372,1338)( 373,1337)( 374,1333)( 375,1334)( 376,1342)
( 377,1341)( 378,1344)( 379,1343)( 380,1339)( 381,1340)( 382,1348)( 383,1347)( 384,1350)
( 385,1349)( 386,1345)( 387,1346)( 388,1327)( 389,1328)( 390,1329)( 391,1330)( 392,1331)
( 393,1332)( 394,1291)( 395,1292)( 396,1293)( 397,1294)( 398,1295)( 399,1296)( 400,1297)
( 401,1298)( 402,1299)( 403,1300)( 404,1301)( 405,1302)( 406,1261)( 407,1262)( 408,1263)
( 409,1264)( 410,1265)( 411,1266)( 412,1267)( 413,1268)( 414,1269)( 415,1270)( 416,1271)
( 417,1272)( 418,1273)( 419,1274)( 420,1275)( 421,1194)( 422,1193)( 423,1192)( 424,1198)
( 425,1197)( 426,1200)( 427,1199)( 428,1195)( 429,1196)( 430,1189)( 431,1190)( 432,1191)
( 433,1186)( 434,1187)( 435,1188)( 436,1218)( 437,1217)( 438,1216)( 439,1222)( 440,1221)
( 441,1224)( 442,1223)( 443,1219)( 444,1220)( 445,1213)( 446,1214)( 447,1215)( 448,1230)
( 449,1229)( 450,1228)( 451,1234)( 452,1233)( 453,1236)( 454,1235)( 455,1231)( 456,1232)
( 457,1225)( 458,1226)( 459,1227)( 460,1246)( 461,1245)( 462,1248)( 463,1247)( 464,1243)
( 465,1244)( 466,1252)( 467,1251)( 468,1254)( 469,1253)( 470,1249)( 471,1250)( 472,1258)
( 473,1257)( 474,1260)( 475,1259)( 476,1255)( 477,1256)( 478,1237)( 479,1238)( 480,1239)
( 481,1240)( 482,1241)( 483,1242)( 484,1201)( 485,1202)( 486,1203)( 487,1204)( 488,1205)
( 489,1206)( 490,1207)( 491,1208)( 492,1209)( 493,1210)( 494,1211)( 495,1212)( 496,1171)
( 497,1172)( 498,1173)( 499,1174)( 500,1175)( 501,1176)( 502,1177)( 503,1178)( 504,1179)
( 505,1180)( 506,1181)( 507,1182)( 508,1183)( 509,1184)( 510,1185)( 511,1156)( 512,1157)
( 513,1158)( 514,1159)( 515,1160)( 516,1161)( 517,1162)( 518,1163)( 519,1164)( 520,1165)
( 521,1166)( 522,1167)( 523,1168)( 524,1169)( 525,1170)( 526,1063)( 527,1064)( 528,1065)
( 529,1060)( 530,1061)( 531,1062)( 532,1053)( 533,1052)( 534,1051)( 535,1058)( 536,1059)
( 537,1055)( 538,1054)( 539,1057)( 540,1056)( 541,1048)( 542,1049)( 543,1050)( 544,1041)
( 545,1040)( 546,1039)( 547,1046)( 548,1047)( 549,1043)( 550,1042)( 551,1045)( 552,1044)
( 553,1026)( 554,1025)( 555,1024)( 556,1023)( 557,1022)( 558,1021)( 559,1037)( 560,1038)
( 561,1034)( 562,1033)( 563,1036)( 564,1035)( 565,1029)( 566,1028)( 567,1027)( 568,1032)
( 569,1031)( 570,1030)( 571,1138)( 572,1139)( 573,1140)( 574,1135)( 575,1136)( 576,1137)
( 577,1128)( 578,1127)( 579,1126)( 580,1133)( 581,1134)( 582,1130)( 583,1129)( 584,1132)
( 585,1131)( 586,1153)( 587,1154)( 588,1155)( 589,1150)( 590,1151)( 591,1152)( 592,1143)
( 593,1142)( 594,1141)( 595,1148)( 596,1149)( 597,1145)( 598,1144)( 599,1147)( 600,1146)
( 601,1111)( 602,1112)( 603,1113)( 604,1104)( 605,1103)( 606,1102)( 607,1109)( 608,1110)
( 609,1106)( 610,1105)( 611,1108)( 612,1107)( 613,1123)( 614,1124)( 615,1125)( 616,1116)
( 617,1115)( 618,1114)( 619,1121)( 620,1122)( 621,1118)( 622,1117)( 623,1120)( 624,1119)
( 625,1074)( 626,1073)( 627,1072)( 628,1077)( 629,1076)( 630,1075)( 631,1069)( 632,1068)
( 633,1071)( 634,1070)( 635,1066)( 636,1067)( 637,1094)( 638,1095)( 639,1091)( 640,1090)
( 641,1093)( 642,1092)( 643,1100)( 644,1101)( 645,1097)( 646,1096)( 647,1099)( 648,1098)
( 649,1081)( 650,1080)( 651,1083)( 652,1082)( 653,1078)( 654,1079)( 655,1087)( 656,1086)
( 657,1089)( 658,1088)( 659,1084)( 660,1085)( 661, 958)( 662, 959)( 663, 960)( 664, 951)
( 665, 950)( 666, 949)( 667, 956)( 668, 957)( 669, 953)( 670, 952)( 671, 955)( 672, 954)
( 673, 936)( 674, 935)( 675, 934)( 676, 933)( 677, 932)( 678, 931)( 679, 947)( 680, 948)
( 681, 944)( 682, 943)( 683, 946)( 684, 945)( 685, 939)( 686, 938)( 687, 937)( 688, 942)
( 689, 941)( 690, 940)( 691,1006)( 692,1007)( 693,1008)( 694, 999)( 695, 998)( 696, 997)
( 697,1004)( 698,1005)( 699,1001)( 700,1000)( 701,1003)( 702,1002)( 703,1018)( 704,1019)
( 705,1020)( 706,1011)( 707,1010)( 708,1009)( 709,1016)( 710,1017)( 711,1013)( 712,1012)
( 713,1015)( 714,1014)( 715, 969)( 716, 968)( 717, 967)( 718, 972)( 719, 971)( 720, 970)
( 721, 964)( 722, 963)( 723, 966)( 724, 965)( 725, 961)( 726, 962)( 727, 989)( 728, 990)
( 729, 986)( 730, 985)( 731, 988)( 732, 987)( 733, 995)( 734, 996)( 735, 992)( 736, 991)
( 737, 994)( 738, 993)( 739, 976)( 740, 975)( 741, 978)( 742, 977)( 743, 973)( 744, 974)
( 745, 982)( 746, 981)( 747, 984)( 748, 983)( 749, 979)( 750, 980)( 751, 858)( 752, 857)
( 753, 856)( 754, 855)( 755, 854)( 756, 853)( 757, 867)( 758, 866)( 759, 865)( 760, 870)
( 761, 869)( 762, 868)( 763, 862)( 764, 861)( 765, 864)( 766, 863)( 767, 859)( 768, 860)
( 769, 846)( 770, 845)( 771, 844)( 772, 850)( 773, 849)( 774, 852)( 775, 851)( 776, 847)
( 777, 848)( 778, 841)( 779, 842)( 780, 843)( 781, 905)( 782, 906)( 783, 902)( 784, 901)
( 785, 904)( 786, 903)( 787, 897)( 788, 896)( 789, 895)( 790, 900)( 791, 899)( 792, 898)
( 793, 923)( 794, 924)( 795, 920)( 796, 919)( 797, 922)( 798, 921)( 799, 929)( 800, 930)
( 801, 926)( 802, 925)( 803, 928)( 804, 927)( 805, 910)( 806, 909)( 807, 912)( 808, 911)
( 809, 907)( 810, 908)( 811, 916)( 812, 915)( 813, 918)( 814, 917)( 815, 913)( 816, 914)
( 817, 876)( 818, 875)( 819, 874)( 820, 880)( 821, 879)( 822, 882)( 823, 881)( 824, 877)
( 825, 878)( 826, 871)( 827, 872)( 828, 873)( 829, 888)( 830, 887)( 831, 886)( 832, 892)
( 833, 891)( 834, 894)( 835, 893)( 836, 889)( 837, 890)( 838, 883)( 839, 884)( 840, 885)
\end{lstlisting}

\noindent
for the colour preserving one, and 

\begin{lstlisting}[language=python]
(   2,  16)(   3,  31)(   5,  19)(   6,  34)(   8,  22)(   9,  37)(  10,  46)(  11,  58)
(  12,  49)(  13,  61)(  14,  70)(  15,  73)(  18,  32)(  21,  35)(  24,  38)(  25,  47)
(  26,  59)(  27,  50)(  28,  62)(  29,  71)(  30,  74)(  40,  48)(  41,  60)(  42,  51)
(  43,  63)(  44,  72)(  45,  75)(  53,  64)(  55,  66)(  56,  76)(  57,  78)(  68,  77)
(  69,  79)(  83,  88)(  85,  90)(  86,  94)(  87,  96)(  92,  95)(  93,  97)( 101, 102)
( 106, 511)( 107, 496)( 108, 406)( 109, 433)( 110, 343)( 111, 148)( 112, 514)( 113, 499)
( 114, 409)( 115, 517)( 116, 502)( 117, 412)( 118, 484)( 119, 394)( 120, 487)( 121, 397)
( 122, 319)( 123, 322)( 124, 430)( 125, 340)( 126, 145)( 127, 445)( 128, 355)( 129, 457)
( 130, 367)( 131, 280)( 132, 292)( 134, 178)( 135, 208)( 136, 423)( 137, 333)( 139, 438)
( 140, 348)( 141, 450)( 142, 360)( 143, 273)( 144, 285)( 146, 171)( 147, 201)( 149, 156)
( 150, 186)( 151, 512)( 152, 497)( 153, 407)( 154, 434)( 155, 344)( 157, 515)( 158, 500)
( 159, 410)( 160, 518)( 161, 503)( 162, 413)( 163, 485)( 164, 395)( 165, 488)( 166, 398)
( 167, 320)( 168, 323)( 169, 431)( 170, 341)( 172, 446)( 173, 356)( 174, 458)( 175, 368)
( 176, 281)( 177, 293)( 180, 209)( 181, 513)( 182, 498)( 183, 408)( 184, 435)( 185, 345)
( 187, 516)( 188, 501)( 189, 411)( 190, 519)( 191, 504)( 192, 414)( 193, 486)( 194, 396)
( 195, 489)( 196, 399)( 197, 321)( 198, 324)( 199, 432)( 200, 342)( 202, 447)( 203, 357)
( 204, 459)( 205, 369)( 206, 282)( 207, 294)( 211, 520)( 212, 505)( 213, 415)( 214, 522)
( 215, 507)( 216, 417)( 217, 490)( 218, 400)( 219, 492)( 220, 402)( 221, 325)( 222, 327)
( 223, 521)( 224, 506)( 225, 416)( 226, 523)( 227, 508)( 228, 418)( 229, 491)( 230, 401)
( 231, 493)( 232, 403)( 233, 326)( 234, 328)( 235, 524)( 236, 509)( 237, 419)( 238, 525)
( 239, 510)( 240, 420)( 241, 494)( 242, 404)( 243, 495)( 244, 405)( 245, 329)( 246, 330)
( 247, 478)( 248, 388)( 249, 480)( 250, 390)( 251, 313)( 252, 315)( 253, 479)( 254, 389)
( 255, 481)( 256, 391)( 257, 314)( 258, 316)( 259, 482)( 260, 392)( 261, 483)( 262, 393)
( 263, 317)( 264, 318)( 266, 267)( 271, 428)( 272, 338)( 274, 443)( 275, 353)( 276, 455)
( 277, 365)( 279, 290)( 283, 429)( 284, 339)( 286, 444)( 287, 354)( 288, 456)( 289, 366)
( 295, 464)( 296, 374)( 297, 470)( 298, 380)( 300, 305)( 301, 465)( 302, 375)( 303, 471)
( 304, 381)( 307, 476)( 308, 386)( 309, 477)( 310, 387)( 331, 422)( 334, 437)( 335, 347)
( 336, 449)( 337, 359)( 346, 425)( 349, 440)( 351, 452)( 352, 362)( 358, 427)( 361, 442)
( 363, 454)( 370, 461)( 372, 467)( 373, 377)( 376, 463)( 378, 469)( 382, 473)( 384, 475)
( 424, 436)( 426, 448)( 441, 451)( 462, 466)( 526,1063)( 527,1138)( 528,1153)( 529,1060)
( 530,1135)( 531,1150)( 532,1048)( 533, 958)( 534, 778)( 535,1111)( 536,1123)( 537,1006)
( 538, 826)( 539,1018)( 540, 838)( 541,1053)( 542,1128)( 543,1143)( 544,1041)( 545, 951)
( 546, 771)( 547,1104)( 548,1116)( 549, 999)( 550, 819)( 551,1011)( 552, 831)( 553,1026)
( 554, 936)( 555, 756)( 556, 858)( 557, 678)( 559,1074)( 560,1077)( 561, 969)( 562, 789)
( 563, 972)( 564, 792)( 565, 867)( 566, 687)( 568, 870)( 569, 690)( 571,1064)( 572,1139)
( 573,1154)( 574,1061)( 575,1136)( 576,1151)( 577,1049)( 578, 959)( 579, 779)( 580,1112)
( 581,1124)( 582,1007)( 583, 827)( 584,1019)( 585, 839)( 586,1065)( 587,1140)( 588,1155)
( 589,1062)( 590,1137)( 591,1152)( 592,1050)( 593, 960)( 594, 780)( 595,1113)( 596,1125)
( 597,1008)( 598, 828)( 599,1020)( 600, 840)( 601,1058)( 602,1133)( 603,1148)( 604,1046)
( 605, 956)( 606, 776)( 607,1109)( 608,1121)( 609,1004)( 610, 824)( 611,1016)( 612, 836)
( 613,1059)( 614,1134)( 615,1149)( 616,1047)( 617, 957)( 618, 777)( 619,1110)( 620,1122)
( 621,1005)( 622, 825)( 623,1017)( 624, 837)( 625,1037)( 626, 947)( 627, 767)( 628,1038)
( 629, 948)( 630, 768)( 631, 905)( 632, 725)( 633, 906)( 634, 726)( 637,1094)( 638,1100)
( 639, 989)( 640, 809)( 641, 995)( 642, 815)( 643,1095)( 644,1101)( 645, 990)( 646, 810)
( 647, 996)( 648, 816)( 649, 923)( 650, 743)( 651, 924)( 652, 744)( 655, 929)( 656, 749)
( 657, 930)( 658, 750)( 661,1052)( 662,1127)( 663,1142)( 664,1040)( 665, 950)( 666, 770)
( 667,1103)( 668,1115)( 669, 998)( 670, 818)( 671,1010)( 672, 830)( 673,1025)( 674, 935)
( 675, 755)( 676, 857)( 679,1073)( 680,1076)( 681, 968)( 682, 788)( 683, 971)( 684, 791)
( 685, 866)( 688, 869)( 691,1055)( 692,1130)( 693,1145)( 694,1043)( 695, 953)( 696, 773)
( 697,1106)( 698,1118)( 699,1001)( 700, 821)( 701,1013)( 702, 833)( 703,1057)( 704,1132)
( 705,1147)( 706,1045)( 707, 955)( 708, 775)( 709,1108)( 710,1120)( 711,1003)( 712, 823)
( 713,1015)( 714, 835)( 715,1034)( 716, 944)( 717, 764)( 718,1036)( 719, 946)( 720, 766)
( 721, 902)( 723, 904)( 727,1091)( 728,1097)( 729, 986)( 730, 806)( 731, 992)( 732, 812)
( 733,1093)( 734,1099)( 735, 988)( 736, 808)( 737, 994)( 738, 814)( 739, 920)( 741, 922)
( 745, 926)( 747, 928)( 751,1023)( 752, 933)( 754, 855)( 757,1029)( 758, 939)( 760,1032)
( 761, 942)( 763, 897)( 765, 900)( 769, 846)( 772, 876)( 774, 888)( 781,1069)( 782,1071)
( 783, 964)( 785, 966)( 787, 862)( 790, 864)( 793,1081)( 794,1083)( 795, 976)( 797, 978)
( 799,1087)( 800,1089)( 801, 982)( 803, 984)( 805, 910)( 807, 916)( 811, 912)( 813, 918)
( 817, 850)( 820, 880)( 822, 892)( 829, 852)( 832, 882)( 834, 894)( 841,1051)( 842,1126)
( 843,1141)( 844,1039)( 845, 949)( 847,1102)( 848,1114)( 849, 997)( 851,1009)( 853,1024)
( 854, 934)( 859,1072)( 860,1075)( 861, 967)( 863, 970)( 871,1054)( 872,1129)( 873,1144)
( 874,1042)( 875, 952)( 877,1105)( 878,1117)( 879,1000)( 881,1012)( 883,1056)( 884,1131)
( 885,1146)( 886,1044)( 887, 954)( 889,1107)( 890,1119)( 891,1002)( 893,1014)( 895,1033)
( 896, 943)( 898,1035)( 899, 945)( 907,1090)( 908,1096)( 909, 985)( 911, 991)( 913,1092)
( 914,1098)( 915, 987)( 917, 993)( 931,1022)( 937,1028)( 940,1031)( 961,1068)( 962,1070)
( 973,1080)( 974,1082)( 979,1086)( 980,1088)(1156,1383)(1157,1473)(1158,1503)(1159,1392)
(1160,1482)(1161,1512)(1162,1395)(1163,1485)(1164,1515)(1165,1548)(1166,1560)(1167,1551)
(1168,1563)(1169,1572)(1170,1575)(1171,1382)(1172,1472)(1173,1502)(1174,1391)(1175,1481)
(1176,1511)(1177,1394)(1178,1484)(1179,1514)(1180,1547)(1181,1559)(1182,1550)(1183,1562)
(1184,1571)(1185,1574)(1186,1380)(1187,1470)(1188,1500)(1189,1371)(1190,1461)(1191,1491)
(1192,1359)(1193,1284)(1195,1413)(1196,1425)(1197,1308)(1198,1218)(1199,1320)(1200,1230)
(1201,1387)(1202,1477)(1203,1507)(1204,1389)(1205,1479)(1206,1509)(1207,1543)(1208,1555)
(1209,1545)(1210,1557)(1211,1567)(1212,1569)(1213,1375)(1214,1465)(1215,1495)(1216,1363)
(1217,1288)(1219,1417)(1220,1429)(1221,1312)(1223,1324)(1224,1234)(1225,1377)(1226,1467)
(1227,1497)(1228,1365)(1229,1290)(1231,1419)(1232,1431)(1233,1314)(1235,1326)(1237,1525)
(1238,1531)(1239,1527)(1240,1533)(1241,1537)(1242,1539)(1243,1441)(1244,1447)(1245,1336)
(1247,1342)(1248,1252)(1249,1443)(1250,1449)(1251,1338)(1253,1344)(1255,1453)(1256,1455)
(1257,1348)(1259,1350)(1261,1381)(1262,1471)(1263,1501)(1264,1390)(1265,1480)(1266,1510)
(1267,1393)(1268,1483)(1269,1513)(1270,1546)(1271,1558)(1272,1549)(1273,1561)(1274,1570)
(1275,1573)(1276,1379)(1277,1469)(1278,1499)(1279,1370)(1280,1460)(1281,1490)(1282,1358)
(1285,1412)(1286,1424)(1287,1307)(1289,1319)(1291,1386)(1292,1476)(1293,1506)(1294,1388)
(1295,1478)(1296,1508)(1297,1542)(1298,1554)(1299,1544)(1300,1556)(1301,1566)(1302,1568)
(1303,1374)(1304,1464)(1305,1494)(1306,1362)(1309,1416)(1310,1428)(1313,1323)(1315,1376)
(1316,1466)(1317,1496)(1318,1364)(1321,1418)(1322,1430)(1327,1524)(1328,1530)(1329,1526)
(1330,1532)(1331,1536)(1332,1538)(1333,1440)(1334,1446)(1337,1341)(1339,1442)(1340,1448)
(1345,1452)(1346,1454)(1351,1378)(1352,1468)(1353,1498)(1354,1369)(1355,1459)(1356,1489)
(1360,1411)(1361,1423)(1367,1456)(1368,1486)(1372,1408)(1373,1420)(1384,1396)(1385,1399)
(1397,1474)(1398,1504)(1400,1475)(1401,1505)(1402,1540)(1403,1552)(1404,1541)(1405,1553)
(1406,1564)(1407,1565)(1409,1462)(1410,1492)(1415,1426)(1421,1463)(1422,1493)(1432,1522)
(1433,1528)(1434,1523)(1435,1529)(1436,1534)(1437,1535)(1439,1444)(1458,1487)(1517,1518)
(1577,1591)(1578,1606)(1580,1594)(1581,1609)(1583,1597)(1584,1612)(1585,1621)(1586,1633)
(1587,1624)(1588,1636)(1589,1645)(1590,1648)(1593,1607)(1596,1610)(1599,1613)(1600,1622)
(1601,1634)(1602,1625)(1603,1637)(1604,1646)(1605,1649)(1615,1623)(1616,1635)(1617,1626)
(1618,1638)(1619,1647)(1620,1650)(1628,1639)(1630,1641)(1631,1651)(1632,1653)(1643,1652)
(1644,1654)(1658,1663)(1660,1665)(1661,1669)(1662,1671)(1667,1670)(1668,1672)(1676,1677)
\end{lstlisting}

\noindent
for the one shuffling colours.
Orbit representatives are then given for the first five colours $0$ to $4$ as follows:

\begin{lstlisting}[language=python]
Colour 0:
(1,511),(2,512),(3,513),(4,514),(5,515),(6,516),(7,517),(8,518),(9,519),(10,520),(11,521),
(12,522),(13,523),(14,524),(15,525),(16,496),(17,497),(18,498),(19,499),(20,500),(21,501),
(22,502),(23,503),(24,504),(25,505),(26,506),(27,507),(28,508),(29,509),(30,510),(31,433),
(32,434),(33,435),(34,430),(35,431),(36,432),(37,423),(38,422),(39,421),(40,428),(41,429),
(42,425),(43,424),(44,427),(45,426),(46,484),(47,485),(48,486),(49,487),(50,488),(51,489),
(52,490),(53,491),(54,492),(55,493),(56,494),(57,495),(58,445),(59,446),(60,447),(61,438),
(62,437),(63,436),(64,443),(65,444),(66,440),(67,439),(68,442),(69,441),(70,457),(71,458),
(72,459),(73,450),(74,449),(75,448),(76,455),(77,456),(78,452),(79,451),(80,454),(81,453),
(82,478),(83,479),(84,480),(85,481),(86,482),(87,483),(88,464),(89,465),(90,461),(91,460),
(92,463),(93,462),(94,470),(95,471),(96,467),(97,466),(98,469),(99,468),(100,476),
(101,477),(102,473),(103,472),(104,475),(105,474),(106,1021),(107,1022),(108,1023),
(109,1024),(110,1025),(111,1026),(112,1027),(113,1028),(114,1029),(115,1030),(116,1031),
(117,1032),(118,1033),(119,1034),(120,1035),(121,1036),(122,1037),(123,1038),(124,1039),
(125,1040),(126,1041),(127,1042),(128,1043),(129,1044),(130,1045),(131,1046),(132,1047),
(133,1048),(134,1049),(135,1050),(136,1051),(137,1052),(138,1053),(139,1054),(140,1055),
(141,1056),(142,1057),(143,1058),(144,1059),(145,1060),(146,1061),(147,1062),(148,1063),
(149,1064),(150,1065),(151,931),(152,932),(153,933),(154,934),(155,935),(156,936),
(157,937),(158,938),(159,939),(160,940),(161,941),(162,942),(163,943),(164,944),(165,945),
(166,946),(167,947),(168,948),(169,949),(170,950),(171,951),(172,952),(173,953),(174,954),
(175,955),(176,956),(177,957),(178,958),(179,959),(180,960),(181,853),(182,854),(183,855),
(184,856),(185,857),(186,858),(187,844),(188,845),(189,846),(190,841),(191,842),(192,843),
(193,849),(194,850),(195,847),(196,848),(197,852),(198,851),(199,865),(200,866),(201,867),
(202,861),(203,862),(204,859),(205,860),(206,864),(207,863),(208,870),(209,869),(210,868),
(211,895),(212,896),(213,897),(214,898),(215,899),(216,900),(217,901),(218,902),(219,903),
(220,904),(221,905),(222,906),(223,874),(224,875),(225,876),(226,871),(227,872),(228,873),
(229,879),(230,880),(231,877),(232,878),(233,882),(234,881),(235,886),(236,887),(237,888),
(238,883),(239,884),(240,885),(241,891),(242,892),(243,889),(244,890),(245,894),(246,893),
(247,919),(248,920),(249,921),(250,922),(251,923),(252,924),(253,909),(254,910),(255,907),
(256,908),(257,912),(258,911),(259,915),(260,916),(261,913),(262,914),(263,918),(264,917),
(265,929),(266,930),(267,926),(268,925),(269,928),(270,927),(271,997),(272,998),(273,999),
(274,1000),(275,1001),(276,1002),(277,1003),(278,1004),(279,1005),(280,1006),(281,1007),
(282,1008),(283,967),(284,968),(285,969),(286,963),(287,964),(288,961),(289,962),
(290,966),(291,965),(292,972),(293,971),(294,970),(295,985),(296,986),(297,987),(298,988),
(299,989),(300,990),(301,975),(302,976),(303,973),(304,974),(305,978),(306,977),(307,981),
(308,982),(309,979),(310,980),(311,984),(312,983),(313,995),(314,996),(315,992),(316,991),
(317,994),(318,993),(319,1018),(320,1019),(321,1020),(322,1011),(323,1010),(324,1009),
(325,1016),(326,1017),(327,1013),(328,1012),(329,1015),(330,1014),(331,1126),(332,1127),
(333,1128),(334,1129),(335,1130),(336,1131),(337,1132),(338,1133),(339,1134),(340,1135),
(341,1136),(342,1137),(343,1138),(344,1139),(345,1140),(346,1102),(347,1103),(348,1104),
(349,1105),(350,1106),(351,1107),(352,1108),(353,1109),(354,1110),(355,1111),(356,1112),
(357,1113),(358,1072),(359,1073),(360,1074),(361,1068),(362,1069),(363,1066),(364,1067),
(365,1071),(366,1070),(367,1077),(368,1076),(369,1075),(370,1090),(371,1091),(372,1092),
(373,1093),(374,1094),(375,1095),(376,1080),(377,1081),(378,1078),(379,1079),(380,1083),
(381,1082),(382,1086),(383,1087),(384,1084),(385,1085),(386,1089),(387,1088),(388,1100),
(389,1101),(390,1097),(391,1096),(392,1099),(393,1098),(394,1123),(395,1124),(396,1125),
(397,1116),(398,1115),(399,1114),(400,1121),(401,1122),(402,1118),(403,1117),(404,1120),
(405,1119),(406,1153),(407,1154),(408,1155),(409,1150),(410,1151),(411,1152),(412,1143),
(413,1142),(414,1141),(415,1148),(416,1149),(417,1145),(418,1144),(419,1147),(420,1146)

\end{lstlisting}

\begin{lstlisting}[language=python]
Colour 1:
(1,511),(2,512),(3,513),(4,514),(5,515),(6,516),(7,517),(8,518),(9,519),(10,520),(11,521),
(12,522),(13,523),(14,524),(15,525),(16,31),(17,32),(18,33),(19,34),(20,35),(21,36),
(22,37),(23,38),(24,39),(25,40),(26,41),(27,42),(28,43),(29,44),(30,45),(46,58),(47,59),
(48,60),(49,61),(50,62),(51,63),(52,64),(53,65),(54,66),(55,67),(56,68),(57,69),(70,73),
(71,74),(72,75),(76,78),(77,79),(80,81),(82,88),(83,89),(84,90),(85,91),(86,92),(87,93),
(94,96),(95,97),(98,99),(100,102),(101,103),(104,105),(106,1021),(107,1022),(108,1023),
(109,1024),(110,1025),(111,1026),(112,1027),(113,1028),(114,1029),(115,1030),(116,1031),
(117,1032),(118,1033),(119,1034),(120,1035),(121,1036),(122,1037),(123,1038),(124,1039),
(125,1040),(126,1041),(127,1042),(128,1043),(129,1044),(130,1045),(131,1046),(132,1047),
(133,1048),(134,1049),(135,1050),(136,1051),(137,1052),(138,1053),(139,1054),(140,1055),
(141,1056),(142,1057),(143,1058),(144,1059),(145,1060),(146,1061),(147,1062),(148,1063),
(149,1064),(150,1065),(151,181),(152,182),(153,183),(154,184),(155,185),(156,186),
(157,187),(158,188),(159,189),(160,190),(161,191),(162,192),(163,193),(164,194),(165,195),
(166,196),(167,197),(168,198),(169,199),(170,200),(171,201),(172,202),(173,203),(174,204),
(175,205),(176,206),(177,207),(178,208),(179,209),(180,210),(211,223),(212,224),(213,225),
(214,226),(215,227),(216,228),(217,229),(218,230),(219,231),(220,232),(221,233),(222,234),
(235,238),(236,239),(237,240),(241,243),(242,244),(245,246),(247,253),(248,254),(249,255),
(250,256),(251,257),(252,258),(259,261),(260,262),(263,264),(265,267),(266,268),(269,270),
(271,283),(272,284),(273,285),(274,286),(275,287),(276,288),(277,289),(278,290),(279,291),
(280,292),(281,293),(282,294),(295,301),(296,302),(297,303),(298,304),(299,305),(300,306),
(307,309),(308,310),(311,312),(313,315),(314,316),(317,318),(319,322),(320,323),(321,324),
(325,327),(326,328),(329,330),(331,421),(332,422),(333,423),(334,424),(335,425),(336,426),
(337,427),(338,428),(339,429),(340,430),(341,431),(342,432),(343,433),(344,434),(345,435),
(346,436),(347,437),(348,438),(349,439),(350,440),(351,441),(352,442),(353,443),(354,444),
(355,445),(356,446),(357,447),(358,448),(359,449),(360,450),(361,451),(362,452),(363,453),
(364,454),(365,455),(366,456),(367,457),(368,458),(369,459),(370,460),(371,461),(372,462),
(373,463),(374,464),(375,465),(376,466),(377,467),(378,468),(379,469),(380,470),(381,471),
(382,472),(383,473),(384,474),(385,475),(386,476),(387,477),(388,478),(389,479),(390,480),
(391,481),(392,482),(393,483),(394,484),(395,485),(396,486),(397,487),(398,488),(399,489),
(400,490),(401,491),(402,492),(403,493),(404,494),(405,495),(406,496),(407,497),(408,498),
(409,499),(410,500),(411,501),(412,502),(413,503),(414,504),(415,505),(416,506),(417,507),
(418,508),(419,509),(420,510),(571,586),(572,587),(573,588),(574,589),(575,590),(576,591),
(577,592),(578,593),(579,594),(580,595),(581,596),(582,597),(583,598),(584,599),(585,600),
(601,613),(602,614),(603,615),(604,616),(605,617),(606,618),(607,619),(608,620),(609,621),
(610,622),(611,623),(612,624),(625,628),(626,629),(627,630),(631,633),(632,634),(635,636),
(637,643),(638,644),(639,645),(640,646),(641,647),(642,648),(649,651),(650,652),(653,654),
(655,657),(656,658),(659,660),(661,841),(662,842),(663,843),(664,844),(665,845),(666,846),
(667,847),(668,848),(669,849),(670,850),(671,851),(672,852),(673,853),(674,854),(675,855),
(676,856),(677,857),(678,858),(679,859),(680,860),(681,861),(682,862),(683,863),(684,864),
(685,865),(686,866),(687,867),(688,868),(689,869),(690,870),(691,871),(692,872),(693,873),
(694,874),(695,875),(696,876),(697,877),(698,878),(699,879),(700,880),(701,881),(702,882),
(703,883),(704,884),(705,885),(706,886),(707,887),(708,888),(709,889),(710,890),(711,891),
(712,892),(713,893),(714,894),(715,895),(716,896),(717,897),(718,898),(719,899),(720,900),
(721,901),(722,902),(723,903),(724,904),(725,905),(726,906),(727,907),(728,908),(729,909),
(730,910),(731,911),(732,912),(733,913),(734,914),(735,915),(736,916),(737,917),(738,918),
(739,919),(740,920),(741,921),(742,922),(743,923),(744,924),(745,925),(746,926),(747,927),
(748,928),(749,929),(750,930)
\end{lstlisting}

\begin{lstlisting}[language=python]
Colour 2:
(1,16),(2,17),(3,18),(4,19),(5,20),(6,21),(7,22),(8,23),(9,24),(10,25),(11,26),(12,27),
(13,28),(14,29),(15,30),(31,406),(32,407),(33,408),(34,409),(35,410),(36,411),(37,412),
(38,413),(39,414),(40,415),(41,416),(42,417),(43,418),(44,419),(45,420),(46,49),(47,50),
(48,51),(52,54),(53,55),(56,57),(58,70),(59,71),(60,72),(61,73),(62,74),(63,75),(64,76),
(65,77),(66,78),(67,79),(68,80),(69,81),(82,84),(83,85),(86,87),(88,94),(89,95),(90,96),
(91,97),(92,98),(93,99),(100,101),(102,104),(103,105),(106,151),(107,152),(108,153),
(109,154),(110,155),(111,156),(112,157),(113,158),(114,159),(115,160),(116,161),(117,162),
(118,163),(119,164),(120,165),(121,166),(122,167),(123,168),(124,169),(125,170),(126,171),
(127,172),(128,173),(129,174),(130,175),(131,176),(132,177),(133,178),(134,179),(135,180),
(136,331),(137,332),(138,333),(139,334),(140,335),(141,336),(142,337),(143,338),(144,339),
(145,340),(146,341),(147,342),(148,343),(149,344),(150,345),(181,751),(182,752),(183,753),
(184,754),(185,755),(186,756),(187,757),(188,758),(189,759),(190,760),(191,761),(192,762),
(193,763),(194,764),(195,765),(196,766),(197,767),(198,768),(199,769),(200,770),(201,771),
(202,772),(203,773),(204,774),(205,775),(206,776),(207,777),(208,778),(209,779),(210,780),
(211,214),(212,215),(213,216),(217,219),(218,220),(221,222),(223,235),(224,236),(225,237),
(226,238),(227,239),(228,240),(229,241),(230,242),(231,243),(232,244),(233,245),(234,246),
(247,249),(248,250),(251,252),(253,259),(254,260),(255,261),(256,262),(257,263),(258,264),
(265,266),(267,269),(268,270),(271,346),(272,347),(273,348),(274,349),(275,350),(276,351),
(277,352),(278,353),(279,354),(280,355),(281,356),(282,357),(283,358),(284,359),(285,360),
(286,361),(287,362),(288,363),(289,364),(290,365),(291,366),(292,367),(293,368),(294,369),
(295,370),(296,371),(297,372),(298,373),(299,374),(300,375),(301,376),(302,377),(303,378),
(304,379),(305,380),(306,381),(307,382),(308,383),(309,384),(310,385),(311,386),(312,387),
(313,388),(314,389),(315,390),(316,391),(317,392),(318,393),(319,394),(320,395),(321,396),
(322,397),(323,398),(324,399),(325,400),(326,401),(327,402),(328,403),(329,404),(330,405),
(421,1141),(422,1142),(423,1143),(424,1144),(425,1145),(426,1146),(427,1147),(428,1148),
(429,1149),(430,1150),(431,1151),(432,1152),(433,1153),(434,1154),(435,1155),(436,448),
(437,449),(438,450),(439,451),(440,452),(441,453),(442,454),(443,455),(444,456),(445,457),
(446,458),(447,459),(460,466),(461,467),(462,468),(463,469),(464,470),(465,471),(472,474),
(473,475),(476,477),(478,480),(479,481),(482,483),(484,487),(485,488),(486,489),(490,492),
(491,493),(494,495),(496,511),(497,512),(498,513),(499,514),(500,515),(501,516),(502,517),
(503,518),(504,519),(505,520),(506,521),(507,522),(508,523),(509,524),(510,525),(526,571),
(527,572),(528,573),(529,574),(530,575),(531,576),(532,577),(533,578),(534,579),(535,580),
(536,581),(537,582),(538,583),(539,584),(540,585),(541,661),(542,662),(543,663),(544,664),
(545,665),(546,666),(547,667),(548,668),(549,669),(550,670),(551,671),(552,672),(553,673),
(554,674),(555,675),(556,676),(557,677),(558,678),(559,679),(560,680),(561,681),(562,682),
(563,683),(564,684),(565,685),(566,686),(567,687),(568,688),(569,689),(570,690),(601,691),
(602,692),(603,693),(604,694),(605,695),(606,696),(607,697),(608,698),(609,699),(610,700),
(611,701),(612,702),(613,703),(614,704),(615,705),(616,706),(617,707),(618,708),(619,709),
(620,710),(621,711),(622,712),(623,713),(624,714),(625,715),(626,716),(627,717),(628,718),
(629,719),(630,720),(631,721),(632,722),(633,723),(634,724),(635,725),(636,726),(637,727),
(638,728),(639,729),(640,730),(641,731),(642,732),(643,733),(644,734),(645,735),(646,736),
(647,737),(648,738),(649,739),(650,740),(651,741),(652,742),(653,743),(654,744),(655,745),
(656,746),(657,747),(658,748),(659,749),(660,750),(781,782),(783,785),(784,786),(787,790),
(788,791),(789,792),(793,794),(795,797),(796,798),(799,800),(801,803),(802,804),(805,811),
(806,812),(807,813),(808,814),(809,815),(810,816),(817,829),(818,830),(819,831),(820,832),
(821,833),(822,834),(823,835),(824,836),(825,837),(826,838),(827,839),(828,840)
\end{lstlisting}

\begin{lstlisting}[language=python]
Colour 3:
(1,148),(2,149),(3,150),(4,7),(5,8),(6,9),(10,12),(11,13),(14,15),(16,31),(17,32),(18,33),
(19,46),(20,47),(21,48),(22,49),(23,50),(24,51),(25,52),(26,53),(27,54),(28,55),(29,56),
(30,57),(34,58),(35,59),(36,60),(37,61),(38,62),(39,63),(40,64),(41,65),(42,66),(43,67),
(44,68),(45,69),(70,319),(71,320),(72,321),(73,322),(74,323),(75,324),(76,325),(77,326),
(78,327),(79,328),(80,329),(81,330),(82,83),(84,86),(85,87),(88,89),(90,92),(91,93),
(94,100),(95,101),(96,102),(97,103),(98,104),(99,105),(106,553),(107,554),(108,555),
(109,556),(110,557),(111,558),(112,115),(113,116),(114,117),(118,120),(119,121),(122,123),
(124,136),(125,137),(126,138),(127,139),(128,140),(129,141),(130,142),(131,143),(132,144),
(133,145),(134,146),(135,147),(151,181),(152,182),(153,183),(154,184),(155,185),(156,186),
(157,211),(158,212),(159,213),(160,214),(161,215),(162,216),(163,217),(164,218),(165,219),
(166,220),(167,221),(168,222),(169,271),(170,272),(171,273),(172,274),(173,275),(174,276),
(175,277),(176,278),(177,279),(178,280),(179,281),(180,282),(187,223),(188,224),(189,225),
(190,226),(191,227),(192,228),(193,229),(194,230),(195,231),(196,232),(197,233),(198,234),
(199,283),(200,284),(201,285),(202,286),(203,287),(204,288),(205,289),(206,290),(207,291),
(208,292),(209,293),(210,294),(235,625),(236,626),(237,627),(238,628),(239,629),(240,630),
(241,631),(242,632),(243,633),(244,634),(245,635),(246,636),(247,295),(248,296),(249,297),
(250,298),(251,299),(252,300),(253,301),(254,302),(255,303),(256,304),(257,305),(258,306),
(259,307),(260,308),(261,309),(262,310),(263,311),(264,312),(265,313),(266,314),(267,315),
(268,316),(269,317),(270,318),(331,346),(332,347),(333,348),(334,349),(335,350),(336,351),
(337,352),(338,353),(339,354),(340,355),(341,356),(342,357),(343,406),(344,407),(345,408),
(358,829),(359,830),(360,831),(361,832),(362,833),(363,834),(364,835),(365,836),(366,837),
(367,838),(368,839),(369,840),(370,372),(371,373),(374,375),(376,382),(377,383),(378,384),
(379,385),(380,386),(381,387),(388,389),(390,392),(391,393),(394,409),(395,410),(396,411),
(397,412),(398,413),(399,414),(400,415),(401,416),(402,417),(403,418),(404,419),(405,420),
(421,436),(422,437),(423,438),(424,439),(425,440),(426,441),(427,442),(428,443),(429,444),
(430,445),(431,446),(432,447),(433,496),(434,497),(435,498),(448,1009),(449,1010),
(450,1011),(451,1012),(452,1013),(453,1014),(454,1015),(455,1016),(456,1017),(457,1018),
(458,1019),(459,1020),(460,462),(461,463),(464,465),(466,472),(467,473),(468,474),
(469,475),(470,476),(471,477),(478,479),(480,482),(481,483),(484,499),(485,500),(486,501),
(487,502),(488,503),(489,504),(490,505),(491,506),(492,507),(493,508),(494,509),(495,510),
(511,1063),(512,1064),(513,1065),(514,517),(515,518),(516,519),(520,522),(521,523),
(524,525),(529,541),(530,542),(531,543),(532,544),(533,545),(534,546),(535,547),(536,548),
(537,549),(538,550),(539,551),(540,552),(559,560),(561,563),(562,564),(565,568),(566,569),
(567,570),(571,586),(572,587),(573,588),(574,601),(575,602),(576,603),(577,604),(578,605),
(579,606),(580,607),(581,608),(582,609),(583,610),(584,611),(585,612),(589,613),(590,614),
(591,615),(592,616),(593,617),(594,618),(595,619),(596,620),(597,621),(598,622),(599,623),
(600,624),(637,638),(639,641),(640,642),(643,644),(645,647),(646,648),(649,655),(650,656),
(651,657),(652,658),(653,659),(654,660),(661,691),(662,692),(663,693),(664,694),(665,695),
(666,696),(667,697),(668,698),(669,699),(670,700),(671,701),(672,702),(673,751),(674,752),
(675,753),(676,754),(677,755),(678,756),(679,781),(680,782),(681,783),(682,784),(683,785),
(684,786),(685,787),(686,788),(687,789),(688,790),(689,791),(690,792),(715,757),(716,758),
(717,759),(718,760),(719,761),(720,762),(721,763),(722,764),(723,765),(724,766),(725,767),
(726,768),(727,793),(728,794),(729,795),(730,796),(731,797),(732,798),(733,799),(734,800),
(735,801),(736,802),(737,803),(738,804),(739,805),(740,806),(741,807),(742,808),(743,809),
(744,810),(745,811),(746,812),(747,813),(748,814),(749,815),(750,816),(769,817),(770,818),
(771,819),(772,820),(773,821),(774,822),(775,823),(776,824),(777,825),(778,826),(779,827),
(780,828)
\end{lstlisting}

\begin{lstlisting}[language=python]
Colour 4:
(1,4),(2,5),(3,6),(7,133),(8,134),(9,135),(10,11),(12,14),(13,15),(16,19),(17,20),(18,21),
(22,178),(23,179),(24,180),(25,26),(27,29),(28,30),(31,34),(32,35),(33,36),(37,208),
(38,209),(39,210),(40,41),(42,44),(43,45),(46,58),(47,59),(48,60),(49,70),(50,71),(51,72),
(52,82),(53,83),(54,84),(55,85),(56,86),(57,87),(61,73),(62,74),(63,75),(64,88),(65,89),
(66,90),(67,91),(68,92),(69,93),(76,94),(77,95),(78,96),(79,97),(80,98),(81,99),(100,265),
(101,266),(102,267),(103,268),(104,269),(105,270),(106,112),(107,113),(108,114),(109,124),
(110,125),(111,126),(115,532),(116,533),(117,534),(118,127),(119,128),(120,129),(121,130),
(122,131),(123,132),(136,568),(137,569),(138,570),(139,141),(140,142),(143,144),(145,148),
(146,149),(147,150),(151,157),(152,158),(153,159),(154,169),(155,170),(156,171),(160,577),
(161,578),(162,579),(163,172),(164,173),(165,174),(166,175),(167,176),(168,177),(181,187),
(182,188),(183,189),(184,199),(185,200),(186,201),(190,592),(191,593),(192,594),(193,202),
(194,203),(195,204),(196,205),(197,206),(198,207),(211,223),(212,224),(213,225),(214,235),
(215,236),(216,237),(217,247),(218,248),(219,249),(220,250),(221,251),(222,252),(226,238),
(227,239),(228,240),(229,253),(230,254),(231,255),(232,256),(233,257),(234,258),(241,259),
(242,260),(243,261),(244,262),(245,263),(246,264),(271,283),(272,284),(273,285),(274,295),
(275,296),(276,297),(277,298),(278,299),(279,300),(280,319),(281,320),(282,321),(286,301),
(287,302),(288,303),(289,304),(290,305),(291,306),(292,322),(293,323),(294,324),(307,655),
(308,656),(309,657),(310,658),(311,659),(312,660),(313,325),(314,326),(315,327),(316,328),
(317,329),(318,330),(331,688),(332,689),(333,690),(334,336),(335,337),(338,339),(340,343),
(341,344),(342,345),(346,358),(347,359),(348,360),(349,370),(350,371),(351,372),(352,373),
(353,374),(354,375),(355,394),(356,395),(357,396),(361,376),(362,377),(363,378),(364,379),
(365,380),(366,381),(367,397),(368,398),(369,399),(382,745),(383,746),(384,747),(385,748),
(386,749),(387,750),(388,400),(389,401),(390,402),(391,403),(392,404),(393,405),(406,409),
(407,410),(408,411),(412,778),(413,779),(414,780),(415,416),(417,419),(418,420),(421,868),
(422,869),(423,870),(424,426),(425,427),(428,429),(430,433),(431,434),(432,435),(436,448),
(437,449),(438,450),(439,460),(440,461),(441,462),(442,463),(443,464),(444,465),(445,484),
(446,485),(447,486),(451,466),(452,467),(453,468),(454,469),(455,470),(456,471),(457,487),
(458,488),(459,489),(472,925),(473,926),(474,927),(475,928),(476,929),(477,930),(478,490),
(479,491),(480,492),(481,493),(482,494),(483,495),(496,499),(497,500),(498,501),(502,958),
(503,959),(504,960),(505,506),(507,509),(508,510),(511,514),(512,515),(513,516),
(517,1048),(518,1049),(519,1050),(520,521),(522,524),(523,525),(526,529),(527,530),
(528,531),(535,536),(537,539),(538,540),(544,553),(545,554),(546,555),(547,559),(548,560),
(549,561),(550,562),(551,563),(552,564),(556,565),(557,566),(558,567),(571,574),(572,575),
(573,576),(580,581),(582,584),(583,585),(586,589),(587,590),(588,591),(595,596),(597,599),
(598,600),(601,613),(602,614),(603,615),(604,625),(605,626),(606,627),(607,637),(608,638),
(609,639),(610,640),(611,641),(612,642),(616,628),(617,629),(618,630),(619,643),(620,644),
(621,645),(622,646),(623,647),(624,648),(631,649),(632,650),(633,651),(634,652),(635,653),
(636,654),(664,673),(665,674),(666,675),(667,679),(668,680),(669,681),(670,682),(671,683),
(672,684),(676,685),(677,686),(678,687),(691,703),(692,704),(693,705),(694,715),(695,716),
(696,717),(697,727),(698,728),(699,729),(700,730),(701,731),(702,732),(706,718),(707,719),
(708,720),(709,733),(710,734),(711,735),(712,736),(713,737),(714,738),(721,739),(722,740),
(723,741),(724,742),(725,743),(726,744),(751,757),(752,758),(753,759),(754,769),(755,770),
(756,771),(763,772),(764,773),(765,774),(766,775),(767,776),(768,777),(781,793),(782,794),
(783,795),(784,796),(785,797),(786,798),(787,817),(788,818),(789,819),(790,829),(791,830),
(792,831),(805,820),(806,821),(807,822),(808,823),(809,824),(810,825),(811,832),(812,833),
(813,834),(814,835),(815,836),(816,837),(826,838),(827,839),(828,840)
\end{lstlisting}


\begin{thebibliography}{99}
\bibitem[1]{AAK2022}
K. Adiprasito, S. Avvakumov, R Karasev, A subexponential size triangulation of $\mathbb{RP}^n$, {\em Combinatorica} {\bf 42} (2022), 1--8.
\bibitem[2]{AM91} 
P. Arnoux and A. Marin, The K\"{u}hnel triangulation of the complex projective plane from the view-point of complex crystallography (part II), {\em Memoirs of Fac. Sc., Kyushu Univ.} Ser. A {\bf 45} (1991), 167--244.
\bibitem[3]{BD94}
B. Bagchi and B. Datta, On K\"{u}hnel's 9-vertex complex projective plane, {\em Geom. Dedicata} {\bf 50} (1994), 1--13.
\bibitem[4]{BD11} 
B. Bagchi and B. Datta, From the icosahedron to natural triangulations of $\mathbb{CP}^{2}$ and $S^{\,2} \times S^{\,2}$, {\em Discrete Comput Geom.} {\bf 46} (2011), 542--560.
\bibitem[5]{BD12} 
B. Bagchi and B. Datta, A triangulation of $\mathbb{C} \mathbb{P}^{3}$ as symmetric cube of ${\rm S}^{2}$, {\em Discrete Comput Geom.} {\bf 48} (2012), 310--329. 
\bibitem[6]{BK92} 
T. F. Banchoff and W. K\"{u}hnel, Equilibrium Triangulations of the Complex Projective Plane, {\em Geom. Dedicata} {\bf 44} (1992), 313--333.
\bibitem[7]{BD14} 
B. Basak and B. Datta, Minimal crystallizations of 3-manifolds, {\em Electronic J. Combin.} {\bf 21\,(1)} (2014), \#P1.61, 25 pp.
\bibitem[8]{BCS88} 
L. J. Billera, R. Cushman and J. A. Sanders, The Stanley decomposition of the harmonic oscillator, {\em Ned. Akad. Wet. Indag. Math.} {\bf 91} (1988), 375--393.
\bibitem[9]{B84}
A. Bj\"{o}rner, Posets, regular CW complexes and Bruhat order, {\em European J. Combin.} {\bf 5} (1984), 7--16. 
\bibitem[10]{Bur11c}
B. Burton, Simplification paths in the Pachner graphs of closed orientable 3-manifold triangulations, arXiv:1110.6080, 2011 
\bibitem[11]{regina}
B. Burton, R. Budney, W. Pettersson et al., Regina: Software for low-dimensional topology, \url{http://regina-normal.github.io/}, 1999--2023
\bibitem[12]{DS2024}
B. Datta and J. Spreer, \url{https://github.com/jspreer/CPn}, 2024.
\bibitem[13]{simpcomp}
F. Effenberger and J. Spreer, simpcomp - A GAP package, Version 2.1.14, \url{https://simpcomp-team.github.io/simpcomp/}, 2009--2021.
\bibitem[14]{ES52} 
S. Eilenberg and N. Steenrod, {\em Foundations of Algebraic Topology}, Princeton University
Press, Princeton, New Jersey, 1952. 
\bibitem[15]{FGG1986}
M. Ferri, C. Gagliardi, L. Grasselli, A graph-theoretical representation of PL-manifolds - a survey on crystallizations. {\em Aequationes Math.} {\bf 31} (1986), 121--141.
\bibitem[16]{GAP}
The GAP Group, GAP -- Groups, Algorithms, and Programming, Version 4.12.2, \url{https://www.gap-system.org}, 2022.
\bibitem[17]{K1986}
W. K\"{u}hnel, Minimal triangulations of Kummer varieties, {\em Abh. Math. Sem. Univ. Hamburg} {\bf 57} (1986), 7--20. 
\bibitem[18]{KB83} 
W. K\"{u}hnel and T. F. Banchoff, The 9-vertex complex projective plane, {\em Math. Intelligencer} {\bf 5 (3)} (1983), 11--22. 
\bibitem[19]{KS23} 
W. K\"{u}hnel, and J. Spreer, Hopf triangulations of spheres and equilibrium triangulations of projective spaces, arXiv.org:2309.12728, 2023
%
\bibitem[20]{LS17}
C. W. Lee and F. Santos, Subdivisions and triangulations of polytopes,  {\em Handbook of Disc. \& Comp. Geom.}, Eds.: J. E. Goodman et al, pp. 415--447,  CRC Press, Boca Raton, FL, 2017. 
%
\bibitem[21]{L03}
F. Lutz, BISTELLAR, \url{http://www.math.TU-Berlin.de/diskregeom/stellar/}, 2003.
\hspace{-1mm}BISTELLAR.
\bibitem[22]{M13}
S. Murai, Face vectors of simplicial cell decompositions of manifolds, {\em Israel J. Math.} {\bf 195} (2013), 187--213.
\bibitem[23]{overflow}
J. Palmieri, G. Kuperberg, D. Speyer, et al., Why is complex projective space triangulable?, \url{https://mathoverflow.net/questions/20664/why-is-complex-projective-space-triangulable}, 2010--2021.
\bibitem[24]{python}
G. Van Rossum, F. Drake, Python 3 Reference Manual, \url{https://www.python.org/}, 2009.
\bibitem[25]{S2014} 
S. Sarkar, Some $\mathbb{Z}_3^n$-equivariant triangulations of $\mathbb{C}P^n$, \texttt{arxiv.or:1405.2568[math.GT]}, 2014, withdrawn 2025. 
\bibitem[26]{S2010}
F. Sergeraert, Triangulations of Complex Projective Spaces, in L. Pardo and A Ib\'a\~nez and J. Garcia and M. G\'omez (editors) {\em Contribuciones cient\'ificas en honor de Mirian Andr\'es G\'omez} (2010), 507--519.
\bibitem[27]{VZ2021}
L. Venturello, H. Zheng, A New Family of Triangulations of $\mathbb{RP}^d$, {\em Combinatorica} {\bf 41} (2021), 127--146.
\end{thebibliography}
\end{document}